\documentclass{amsart}

\usepackage{amssymb}

\usepackage{graphicx}
\usepackage{comment}

\theoremstyle{theorem}
\newtheorem{thm}{Theorem}

\theoremstyle{definition}
\newtheorem{rem}[thm]{Remark}

\newtheorem{lem}[thm]{Lemma}
\newtheorem{defn}[thm]{Definition}

\newtheorem{cor}[thm]{Corollary}
\newtheorem{claim}[thm]{Claim}

\newcommand{\Q}{\mathbb{Q}}

\newcommand{\N}{\mathbb{N}}
\newcommand{\pr}[1]{\left ( #1 \right )}
\newcommand{\abs}[1]{\left | #1 \right |}
\newcommand{\br}[1]{\left \{ #1 \right \}}
\newcommand{\bra}[1]{\left [ #1 \right ]}

\newcommand{\sE}{\mathcal{E}}
\newcommand{\sB}{\mathcal{B}}

\newcommand{\bu}{\bar{u}}

\newcommand{\nor}{\mathcal{N}}
\newcommand{\norp}{{\mathcal{N}}^\perp}

\newcommand{\loe}{{\beta}^{-}}
\newcommand{\upe}{{\beta}^{+}}

\newcommand{\bs}{\boldsymbol{\Sigma}}
\newcommand{\bp}{\boldsymbol{\Pi}}
\newcommand{\bd}{\boldsymbol{\Delta}}

\newcommand{\res}{\restriction}
\newcommand{\conc}{{}^\smallfrown}
\newcommand{\ds}{\displaystyle}
\newcommand{\MM}{\mathcal{M}}
\newcommand{\GG}{\mathcal{G}}

\author[D. Airey]{Dylan Airey}
\address[D. Airey]{
Department of Mathematics, University of Texas at Austin, 2515 Speedway, Austin, TX 78712-1202, USA}
\email{dylan.airey@utexas.edu}
\author[S. Jackson]{Steve Jackson}
\address[S. Jackson]{Department of Mathematics, University of North Texas, General Academics Building 435, 1155 Union Circle,  \#311430, Denton, TX 76203-5017, USA}
\email{stephen.jackson@unt.edu}
\author[B. Mance]{Bill Mance}
\address[B. Mance]{Institute of Mathematics of Polish Academy of Science, \'{S}niadeckich 8, 00-656 Warsaw, Poland}
\email{Bill.A.Mance@gmail.com}

\title{Some complexity results in the theory of normal numbers}

\begin{document}

\begin{abstract}
Let $\nor(b)$ be the set of real numbers which are normal to base $b$. 
A well-known result of H. Ki and T. Linton~\cite{KiLinton} is that $\nor(b)$
is $\bp^0_3$-complete. We show that the set $\norp(b)$ of reals
 which preserve $\nor(b)$ under addition is also $\bp^0_3$-complete. 
We use the characteriztion of $\norp(b)$ given by G. Rauzy in terms of
an entropy-like quantity called the \textit{noise}. It follows from our results that no further 
characteriztion theorems could result in a still better bound on the complexity of $\norp(b)$.
We compute the exact descriptive complexity
of other naturally occurring sets associated with noise. One of these is complete 
at the $\bp^0_4$ level. Finally, we get upper and lower bounds on the Hausdorff dimension
of the level sets associated with the noise.
\end{abstract}

\maketitle

\section{Introduction}

Let $b\geq 2$ be a positive integer. Every real number $x$ has a base $b$ expansion 
$x=\lfloor x\rfloor +\sum_{n=0}^\infty \frac{c_n}{b^n}$, and this expansion is unique if 
we adopt the convention that a tail of the coefficients $c_n$ cannot be equal to $b-1$. 
Recall $x$ is {\it $b$-normal} if for every $B=(i_0,\dots,i_{\ell-1})\in b^{<\omega}$
we have that $\lim_{N \to \infty} \frac{1}{N} |I(x,B,N)| =\frac{1}{b^\ell}$, where 
$I(x,B,N)=\{ i<N \colon (c_i,c_{i+1},\dots,c_{i+\ell-1})=B\}$. For a real number $r$, 
define real functions $\pi_r$ and $\sigma_r$ by $\pi_r(x)=rx$ and $\sigma_r(x)=r+x$. 
We let $\nor(b)$
denote the set of reals $x$ which are normal to base $b$. We let 
$$
\norp(b)
= \{ y \colon \forall x \in \nor(b)\ (x+y)\in \nor(b)\}.
$$ 

\subsection{Normality preserving functions}
Let $f$ be a function from $\mathbb{R}$ to $\mathbb{R}$.  We say that
$f$ {\it preserves $b$-normality} if $f(\nor(b)) \subseteq \nor(b))$.
We can make a similar definition for preserving normality with respect
to the regular continued fraction expansion, $\beta$-expansions,
Cantor series expansions, the L\"uroth series expansion, etc.

Several authors have studied $b$-normality preserving functions.  They
naturally arise in H. Furstenberg's work on disjointness in ergodic
theory\cite{FurstenbergDisjoint}.  V. N. Agafonov
\cite{AgafonovNormal}, T. Kamae \cite{Kamae}, T. Kamae and B. Weiss
\cite{KamaeWeiss}, and W. Merkle and J. Reimann \cite{MerkleReimann}
studied $b$-normality preserving selection rules.  The situation for
continued fractions is very different.  Let $[a_1,a_2,a_3,\ldots]$ be
normal with respect to the continued fraction expansion.  B. Heersink
and J. Vandehey \cite{HeersinkVandehey} recently proved that for any
integers $m \geq 2, k \geq 1$, the continued fraction
$[a_k,a_{m+k},a_{2m+k},a_{3m+k},\ldots]$ is never normal with respect
to the continued fraction expansion.

In 1949 D. D. Wall proved in his Ph.D. thesis \cite{Wall} that for
non-zero rational $r$ the function $\pi_r$ is $b$-normality
preserving for all $b$ and that the function $\sigma_r$ is
$b$-normality preserving for all $b$ whenever $r$ is rational.  These
results were also independently proven by K. T. Chang in 1976
\cite{ChangNormal}.  D. D. Wall's method relies on the well known
characterization that a real number $x$ is normal in base $b$ if and
only if the sequence $(b^nx)$ is uniformly distributed mod
$1$\cite{KuN}.

D. Doty, J. H. Lutz, and S. Nandakumar took a substantially different
approach from D. D. Wall and strengthened his result.  They proved in
\cite{LutzNormalityPreserves} that for every real number $x$ and every
non-zero rational number $r$ the $b$-ary expansions of $x, \pi_r(x),$
and $\sigma_r(x)$ all have the same finite-state dimension and the
same finite-state strong dimension.  It follows that $\pi_r$ and
$\sigma_r$ preserve $b$-normality.  It should be noted that their
proof uses different methods from those used by D. D. Wall and is
unlikely to be proven using similar machinery.

C. Aistleitner generalized D. D. Wall's result on $\sigma_r$ in
\cite{AistleitnerNormalPreserves}.  Suppose that $q$ is a rational
number and that the digits of the $b$-ary expansion of $z$ are
non-zero on a set of indices of density zero.  He proved that the
function $\sigma_{qz}$ is $b$-normality preserving.  G. Rauzy obtained
a complete characterization of $\norp(b)$ in
\cite{RauzyNormalPreserving}. M. Bernay used this characterization to
prove that $\Sigma_b$ has zero Hausdorff dimension \cite{Bernay}.  One of the 
main results of this paper, stated in Corollary~\ref{maincor},  is to obtain an exact 
determination of the descriptive set theoretic complexity 
of $\norp(b)$. A significance of this is explained at 
the end of \S\ref{intro:results} below.

M. Mend\'{e}s France asked in \cite{X} if the function $\pi_r$
preserves simple normality with respect to the regular continued
fraction for every non-zero rational $r$.  This was recently settled
by J. Vandehey \cite{VandeheyRatMult} who showed that $\frac{ax+b}
{cx+d}$ is normal with respect to the continued fraction when $x$ is
normal with respect to the continued fraction expansion and integers
$a, b, c, $and $d$ satisfy $ad-bc \neq 0$.  Work was also done on the
normality preserving properties of the functions $\pi_r$ and
$\sigma_r$ for the Cantor series expansions by the first and third
author in \cite{AireyManceNormalPreserves} and additonally with
J. Vandehey in \cite{AireyManceVandehey}.  However, these functions
are not well understood in this context.

\subsection{Descriptive Complexity}

In any topological space $X$, the collection of Borel sets $\sB(X)$ is the smallest 
$\sigma$-algebra containing the open sets. They are stratified into levels,
the Borel hierarchy, by defining $\bs^0_1=$ the open sets, $\bp^0_1=
\neg \bs^0_1= \{ X-A\colon A \in \bs^0_1\}=$ the closed sets, 
and for $\alpha<\omega_1$ we let $\bs^0_\alpha$ be the collection of 
countable unions $A=\bigcup_n A_n$ where each $A_n \in \bp^0_{\alpha_n}$
for some $\alpha_n<\alpha$. We also let $\bp^0_\alpha=\neg \bs^0_\alpha$. 
Alternatively, $A\in \bp^0_\alpha$ if $A=\bigcap_n A_n$ where 
$A_n\in \bs^0_{\alpha_n}$ where each $\alpha_n<\alpha$. 
We also set $\bd^0_\alpha= \bs^0_\alpha \cap \bs^0_\alpha$, in particular 
$\bd^0_1$ is the collection of clopen sets. 
For any topological space, $\sB(X)=\bigcup_{\alpha<\omega_1} \bs^0_\alpha=
\bigcup_{\alpha<\omega_1}\bp^0_\alpha$.  All of the collections 
$\bd^0_\alpha$, $\bs^0_\alpha$, $\bp^0_\alpha$ are pointclasses, that is, they are closed 
under inverse images of continuous functions. 
A basic fact (see \cite{Kechris})
is that for any uncountable Polish space $X$, there is no collapse 
in the levels of the Borel hierarchy, that is, all the various 
pointclasses $\bd^0_\alpha$, $\bs^0_\alpha$, $\bp^0_\alpha$, for $\alpha <\omega_1$,
are all distinct. Thus, these levels of the Borel hierarch can be used
to calibrate the descriptive complexity of a set. We say a set 
$A\subseteq X$ is $\bs^0_\alpha$ (resp.\ $\bp^0_\alpha$) {\em hard} 
if $A \notin \bp^0_\alpha$ (resp.\ $A\notin \bs^0_\alpha$). This says $A$ is 
``no simpler'' than a $\bs^0_\alpha$ set. We say $A$ is $\bs^0_\alpha$-{\em complete}
if $A\in \bs^0_\alpha-\bp^0_\alpha$, that is, $A \in \bs^0_\alpha$ and 
$A$ is $\bs^0_\alpha$ hard. This says $A$ is exactly at the complexity level 
$\bs^0_\alpha$. Likewise, $A$ is $\bp^0_\alpha$-complete if $A\in \bp^0_\alpha-\bs^0_\alpha$.

A fundamental result of Suslin (see \cite{Kechris}) says that in any 
Polish space $\sB(X)=\bd^1_1=\bs^1_1 \cap \bp^1_1$, where 
$\bp^1_1=\neg \bs^1_1$, and $\bs^1_1$ is the pointclass of continuous images 
of Borel sets. Equivalently, $A \in \bs^1_1$ iff 
$A$ can be written as $x \in a \leftrightarrow \exists y\ (x,y)\in B$ 
where $B \subseteq X\times Y$ is Borel (for some Polish space $Y$). 
Similarly, $A\in \bp^1_1$ iff it is of the form $x\in A \leftrightarrow 
\forall y\ (x,y)\in B$ for a Borel $B$. The $\bs^1_1$ sets are also called the {\em analytic}
sets, and $\bp^1_1$ the {\em co-analytic sets}. We also have 
$\bs^1_1\neq \bp^1_1$ for any uncountable Polish space.

H.\ Ki and T.\ Linton \cite{KiLinton} proved that the set $\nor(b)$ is
$\bp^0_3(\mathbb{R})$-complete.  Further work was done by V.\ Becher,
P.\ A.\ Heiber, and T.\ A.\ Slaman \cite{BecherHeiberSlamanAbsNormal} who
settled a conjecture of A.\ S.\ Kechris by showing that the set of
absolutely normal numbers is $\bp^0_3(\mathbb{R})$-complete.
Furthermore, V.\ Becher and T.\ A.\ Slaman \cite{BecherSlamanNormal}
proved that the set of numbers normal in at least one base is
$\bs^0_4(\mathbb{R})$-complete.

K.\ Beros considered sets involving normal numbers in the difference
heirarchy in \cite{BerosDifferenceSet}.  Let $\nor_k(b)$ be the set of
numbers normal of order $k$ in base $b$.  He proved that for $b\geq 2$
and $s>r\geq 1$, the set $\nor_r(b) \backslash \nor_s(b)$ is
$\mathcal{D}_2(\bp_3^0)$-complete (see \cite{Kechris} for a definition
of the difference hierarchy).  Additionally, the set $\bigcup_k
\nor_{2k+1}(2) \backslash \nor_{2k+2}(2)$ is shown to be
$\mathcal{D}_\omega(\bp_3^0)$-complete.

\subsection{Results of this paper} \label{intro:results}

We are interested in determining the exact descriptive set theoretic complexity of 
$\norp(b)$ and some related sets. The definition of $\norp(b)$ 
shows that $\norp(b)$ is $\bp^1_1$, since it involves a universal quantification. 
It is not immediately clear if $\norp(b)$ is a Borel set, but this in fact 
follows from a result of Rauzy. 
Specifically, Rauzy \cite{RauzyNormalPreserving} characterized $\norp(b)$
in terms of an entropy-like condition called the {\em noise}. 
We recall this condition and associated notation
from \cite{RauzyNormalPreserving}. For any positive integer length $\ell$, let $\sE_\ell$
denote the set of functions from $b^\ell$ to $b$. 
We call an $E\in \sE_\ell$ a {\em block function} of width $\ell$. 
As in \cite{RauzyNormalPreserving}
we set, for $x \in \mathbb{R}$,  
\[
\beta_\ell(x,N)= \inf_{E\in \sE_\ell} \frac{1}{N} \sum_{n<N} \inf \{
1, |c_n-E(c_{n+1},\dots,c_{n+\ell})|\},
\]
where $c_0,c_1,\dots$ is the (fractional part) of the base $b$ expansion of $x$. 

\noindent
We also let for $E\in \sE$ 
\[
\beta_E(x,N)= \frac{1}{N} \sum_{n<N} \inf \{
1, |c_n-E(c_{n+1},\dots,c_{n+\ell})|\}.
\]

We then define the lower and upper noises $\loe(x)$, $\upe(x)$ 
of $x$ by: 
$$\loe(x)=\lim_{\ell \to \infty} \loe_\ell(x),$$ 
where 
$$\loe_\ell(x)=\liminf_{N\to \infty} \beta_\ell(x,N).$$ 
The upper entropy
$\upe(x)$ is defined similarly using 
$$\upe(x)=\lim_{\ell\to \infty} \upe_\ell(x)$$
where 
$$\upe_\ell(x)=\limsup_{N\to \infty} \beta_\ell (x,N).$$ 
For a fixed $E\in \sE$
we also let 
$$
\loe_E(x)=\liminf_{N\to \infty} \beta_E(x,N),$$ 
and similarly 
for $\upe_E(x)$.

Rauzy showed that $x\in \nor(b)$ iff  it has the maximal possible noise
in that $\loe(x)=\frac{b-1}{b}$. Furthermore, $x \in \norp(b)$ iff it has
minmal possible noise in that $\upe(x)=0$.

It is therefore natural to ask for any $s\in [0,\frac{b-1}{b}]$, what are the complexities 
of the lower and upper noise sets associated to $s$. Specifically, we introduce the 
following four sets. 

\begin{defn}
Let $s \in [0,\frac{b-1}{b}]$. Let 
\begin{equation}
\begin{split}
A_1(s)&= \{ x \colon \loe(x) \leq s\},\quad
A_2(s)=\{ x \colon \loe(x) \geq s\} \\
A_3(s)&=\{x \colon \upe(x) \leq s\},\quad
A_4(s)=\{ x \colon \upe(x) \geq s\}
\end{split}
\end{equation}
Finally, we let 
\begin{align*}
L(s)&=A_1(s)\cap A_2(s)= \{ x \colon \loe(x) =s\}\\
U(s)&= A_3(s)\cap A_4(s)=\{ x \colon \upe(x)=s\}.
\end{align*}
\end{defn}

Thus, $\nor(b)= L(\frac{b-1}{b})$, and $\norp(b)=U(0)$. The Ki-Linton result shows that
$\nor(b)$, and thus $L(\frac{b-1}{b})$  is $\bp^0_3$-complete for any 
base $b$. Recall also the Becher-Slaman result which shows that the set of
reals which are normal to some base $b$ forms a $\bs^0_4$-complete set.

We  have the following complexity results.

\begin{thm} \label{mt}
For any $s \in [0,\frac{b-1}{b})$, the set 
$A_1(s)$ is $\bp^0_4$-complete and the set $A_3(s)$ is $\bp^0_3$-complete.
For any $s \in (0,\frac{b-1}{b} ]$, 
the set $A_2(s)$  is $\bp^0_3$-complete, and the set $A_4(s)$ is $\bp^0_2$-complete. 
For $s\in (0,\frac{b-1}{b})$, the set $L(s)$ is $\bp^0_4$-complete, and the set $U(s)$ is 
$\bp^0_3$-complete. 
\end{thm}

As a corollary we obtain the Ki-Linton result as well as the determination of the
exact complexity of $\norp(b)$.

\begin{cor} \label{maincor}
The sets $\nor(b)$ and $\norp(b)$ are both $\bp^0_3$-complete.
\end{cor}

\begin{proof}
We have $x \in \nor(b)$ iff $x \in A_2(\frac{b-1}{b})$, and $x \in \norp(b)$
iff $x \in A_3(0)$, so the result follows immeduately from Theorem~\ref{mt}.
\end{proof}

\begin{rem}
In defining the noise, it is sometimes convenient to use the minor variation 
\[
\beta_E(x,N)= \frac{1}{N} \sum_{\ell<n<N} \inf \{
1, |c_n-\varphi(c_{n-\ell},\dots,c_{n-1})|\},
\]
that is, the block function predicts the next digit rather than the previous digit.
In this case we must start the sum at $\ell+1$ rather than $1$, but this does 
not affect any of the limits used in defining $\loe(x)$ or $\upe(x)$.
\end {rem}

\begin{rem}
In proving Theorem~\ref{mt} we will work with $x$ in the sequence space $X=b^\omega\cap G$
where $G$ is the set of $x \in b^\omega$ which do not end in a tail of $b-1$'s. 
This is a Polish space as $b^\omega$ is a
compact Polish space (with the usual product of discrete topologies on $\{0,1,\dots,b-1\}$) 
and $G$ is a $G_\delta$ (that is, $\bp^0_2$) subset of $b^\omega$. 
This is permissible
as the map $f \colon X \to \mathbb{R}$ given by 
$f(c_0,c_1,\dots)=\sum \frac{c_n}{b^n}$ is continuous. So, for example, 
given that $B_3(s)$ is $\bp^0_3$-complete, where $B_3(s)\subseteq X$ 
is defined as $A_3(s)$, except we consider directly $x \in b^\omega$, then it follows
that $A_3(s)$ is $\bp^0_3$-complete. For if $A_3(s)$ were in $\bs^0_3$, then so would 
be $B_3(s)$ since $x \in B_3(s)\leftrightarrow f(x)\in A_3(s)$, that is, 
$B_3(s)=f^{-1} (A_3(s))$.
\end{rem}

We remark on the significance of complexity classifications such 
Theorem~\ref{mt}. Aside from the intrinsic interest to descriptive set theory, 
results of this form can be viewed as ruling out the existence of further theorems
which would reduce the complexity of the sets. For example, Rauzy's theorem 
reduces the complexity of $\norp(b)$ from $\bp^1_1$ to $\bp^0_3$. The fact that 
$A_3(0)$ is $\bp^0_3$-complete tells us that there cannot be other theorems which 
result in a yet simpler characterization of $\norp(b)$. 

Lastly, we wish to approximate the Hausdorff dimension of the sets $A_i(s), U(s)$, and $L(s)$.  Put $H(s) = -s\log s - (1-s) \log (1-s)$.

\begin{thm}\label{HD}
For $s \in \bra{0,\frac{b-1}{b}}$ we have
\begin{align*}
&\dim_H(A_1(s)) = 1 \\ &\dim_H(A_2(s)) = 1 \\
\frac{1}{\log b} H(s) + \frac{\log(b-1)}{\log b} s \leq &\dim_H(A_3(s)) \leq \frac{1}{\log b} H(s) + s \\ &\dim_H(A_4(s)) = 1.
\end{align*}
Furthermore
\begin{align*}
\frac{1}{\log b} H(s) + \frac{\log(b-1)}{\log b} s \leq &\dim_H(U(s)) \leq  \frac{1}{\log b} H(s) + s \\
& \dim_H(L(s)) = 1.
\end{align*}
\end{thm}

\section{A property of noise}

Before proving the main theorem in \S\ref{sec:proofmt} we show a property of the noise
which shows that one must be careful when attempting to construct reals
with a desired lower-bound for the noise. If we have a set $A \subseteq \omega$ with density $d$,
then for almost all $x \in b^\omega$, if $x_A$ is the result of copying $x$ to the set $A$ and 
taking value $0$ off of $A$, then we easily have that $\beta(x_A)=d(\frac{b-1}{b})$. 
The next lemma shows that it might be possible to lower the noise by taking non-zero
values off of $A$.

\begin{thm} \label{ln}
There is a periodic set $A \subseteq \omega$, say with period $p$ and
density $\frac{d}{p}$, such that if $u\in b^\omega$ satisfies 
$u\res (\omega-A)=0$ and $\loe(u)=d \left( \frac{b-1}{b}\right)$,
then there is a $v \in b^\omega$ with $v\res A=u\res A$ and 
$\upe(v)< d \left( \frac{b-1}{b}\right)$.
\end{thm}

\begin{proof}
Fix a positive integer $k$, and let $\ell > 10 k b^k$. Let $A$ be the set with 
period $\ell+k$ with $A\cap [0, \ell+k-1)=\{ \ell,\ell+1,\dots,\ell+k-1\}$. 
Supppose  $u \in b^\omega$ with $ u \res (\omega-A)=0$ and 
$\loe(u)= d\left( \frac{b-1}{b}\right)$, where $d= \frac{k}{\ell+k}$ is the density of $A$. 
We show that there is a $v \in b^\omega$ with $v \res A= u\res A$ with 
$\upe(v) <\loe(u)$.

Consider the $n$th block $B_n$ of digits of length $\ell+k$, that is $B_n=[ n(\ell+k), (n+1)(\ell+k) )$. 
Let $B'_n=[ n(\ell+k)+\ell, (n+1)(\ell+k) )$ be the last $k$ digits of the block $B_n$.
The idea of the following argument is simply to code $u \res B'_n$ via the position of 
a single digit in  $B_n-B'_n$. If $s \in b^k$, let $b(s)\in 2^p$, where $p={\lceil k \log_2(b) \rceil +1}$,
be the binary repsesentation of the integer represented by the base $b$ expansion $s$, where
we put the least significant digits first. Let $c(s)\in 2^{5p}$ be the result of translating the 
digits of $b(s)$ according to $0 \mapsto 11001$, $1 \mapsto 11011$. 
Note that $c(s)$ will never have more than $4$ consecutive $1$'s. 

Let $w=2(5p+k+8)$. We will describe a particular block function 
$E \colon b^w\to b$ and $v \in b^\omega$ as above. We describe $v \res B_n-B'_n$, which only depends on 
$u \res B'_n$. We say a sequence $t \in 2^{<\omega}$ is {\em canonical} if it is a subsequence of 
a sequence of the form $t_0 \conc t_1 \conc \cdots \conc t_r$ where 
$t_i= a_i \conc 0 \conc b_i \conc 0\conc 11111$ where $a_i=c(b_i)$ and $b_i$ is the $k$-digit base $b$
expansion of $i$ (with least significant digit first).  We assume here that $r \leq b^k$. 
Thus, a canonical sequence is a way of ``counting'' from $0$ to $r$.

The block function $E$ operates as follows. If $s \in b^w$ is the constant $0$ sequence, then 
$E(s)=0$. If the first $1$ (reading from the left) in $s$ occurs at position $p$ to the 
right of the midpoint of $s$, 
then $E(s)$ is the next digit of a canonical sequence starting at the position of this $1$,
that is, $E(s)=t(w-p)$ for a canonical sequence $t$ of length $>w$. If $p<\frac{w}{2}$, 
then $E$ checks to see if there is a sequence of $5$ consecutive $1$'s in $s$. If not, 
$E(s)=0$. If so, let $q$ be the position which is the start of the first such sequence. 
Note that $w$ is large enough so that if $q<\frac{w}{2}$, then reading to right from $q$ 
there are enough positions in $s$ to see a complete ``cycle'' of a canonical sequence 
(that is, some $t_i$). Likewise if $q>\frac{w}{2}$ there is enough room to 
see a cycle to the left of $q$. Then $E$ checks to see if the positions to 
the right (left if $q >\frac{w}{2}$) of $q$ give a cycle of a canonical sequence. 
If so, we check to see if $s$ is the subsequence of the corresponding 
canonical sequence. If so, $E(s)$ is the next digit of this canonical sequence, and if not
we set $E(s)=0$. This completes the definition of $E$. 

We now define $v \res B_n-B'_n$. This consists of a canonical sequence $t$ 
starting at the unique position $q \in B_n-B'_n$ so that the $k$ digits in $B'_n$ 
correspond to the substring $b_i$ of some $t_i$ in $t$. This completes the definition of $v$.

We claim that for each $q \in B_n$, if we let $u_q= u \res [q-w,q-1]$, then 
\[
| \{ q \in B_n\colon E(u_q)\neq u(q) \}| \leq 2.
\]
To see this, first note that $E$ will predict a $0$ at the start of the block $B_n$ 
(corresponding to the $0$ after a $b_i$), which is 
the correct value, and then predict a $1$ (corresponding to the start of a $11111$ sequence) 
at the next position, which is incorrect; note that  
$B_n$ is large enough so that the first $1$ in it is far to the right of the start of the block. 
Given the two initial $0$'s in $B_n$, $E$ will continue to predict $0$ until $E$ 
reaches the point $q \in B_n$
which is the first $1$ (as $E$ cannot find the $5$ consecutive $1$'s it needs to consider outputting
a non-zero value). At position $q$, $E$ will also predict a $0$, which is incorrect. After position $q$, 
$E$ will make correct predictions through the end of $B'_n$, as $u \res [q, q']$ is canonical, where 
$q'=(n+1)(\ell+k)-1$ is the last position of $B'_n$.

Thus, $\upe(v)\leq \frac{2}{\ell+k}$, while $\loe(u) = d\left(\frac{b-1}{b}\right) = \frac{k}{l+k} \frac{b-1}{b}$. 
So, for $k \frac{b-1}{b}>2$, we may choose $\ell$ large enough so that the above construction of $v$ 
works, and we then have $\upe(v) \leq \loe(u)$. 
\end{proof}

\section{Proof of theorem~\ref{mt}} \label{sec:proofmt}
The upper bounds for the complexities of the sets of Theorem~\ref{mt}
all follow by straightforward computations from the definitions of these sets.
For example, consider $A_1(s)$. We  have 
\begin{equation*}
\begin{split}
x \in A_1(s)& \leftrightarrow 
\forall \epsilon \in \Q^+\ \exists m\ \forall \ell \geq m\ \loe_\ell(x)\leq \epsilon\\
&\leftrightarrow \forall \epsilon \in \Q^+\ \exists m\ \forall \ell \geq m\ 
\forall k\ \exists N \geq k\ (\beta_\ell(x,N)\leq \epsilon)
\end{split}
\end{equation*}

\noindent
Since for fixed $\epsilon, \ell, N$ the set $\{ x\colon \beta_\ell(x,N)\leq \epsilon\}$
is clopen, this shows that $A_1(s)\in \bp^0_4$.

The following lemma and its proof will be used several times in the proofs for 
the lower bounds on the complexities in Theorem~\ref{mt}.

\begin{lem} \label{tel}
Let $A \subseteq \omega$ be an infinite set with upper density $d$, and
let $y \in b^\omega$. Then for almost all
$x \in b^\omega$, if $x' \in b^\omega$ is defined by 
\begin{equation*}
x'(n)=\begin{cases}
x(k) &\text{if } n \text{ is the $k$th element of } A \\
y(k) & \text{if } n \text{ is the $k$th element of } \omega-A
\end{cases},
\end{equation*}
\noindent
then $\upe(x')\geq  d \left( \frac{b-1}{b}\right)$.
\end{lem}

\begin{proof}
Fix $\epsilon_n=\frac{1}{n^2}$, and fix a sufficiently fast growing sequence 
$n_0 < n_1 < \cdots$ (we will specify how fast the sequence needs to grow below). 
We also choose the $n_k$ such that $d(A, n_k)\geq d-\epsilon_k$, where 
$d(A, n)=\frac{| A\cap n|}{n}$ is the density of $A$ among $\{0,\dots,n-1\}$. 
Consider the block of integers $B_k=[n_{k-1},n_k)$. Let $m_k=|A \cap B_k|$,
so $m_k\geq (d-\epsilon_k) n_k-n_{k-1}$. Let $l_k=b^{b^k}$ be the number of block functions
of width $k$. Consider one of these block functions $E\colon b^k\to b$. Consider the function 
$\tau \colon b^{m_k}\to b^{m_k}$ defined as follows. If $s \in b^{m_k}$, let $s'$ be the result
of copying $s$ to $B_k \cap A$ and copying $y$ to $B_k-A$ (that is, for $k \in B_k-A$, 
$s'(k)=y(\ell)$ where $k$ is the $\ell$th element of $\omega-A$). 
If $p_i$ is the $i$th 
element of $B_k\cap A$, let $\tau(s)(i)=E(s'\res [p_i-k, p_i-1])-s'(p_i) \mod b$.  
So, $\tau(s)\in b^{m_k}$. Clearly $\tau$ is a bijection from $b^{m_k}$ to $b^{m_k}$. 
So, the number of $s\in b^{m_k}$ for which there are exactly $a$ many $i<m_k$
such that $E(s'\res [p_i-k, p_i-1])\neq s'(p_i)$ is equal to the number of $s \in b^{m_k}$
such that $s$ has exactly $a$ non-zero digits. 
So, the number $e(m_k)$ of $s \in b^{m_k}$ such that 
\[ 
\abs{\{ i <m_k\colon E(s'\res [p_i-k, p_i-1]) \neq s'(p_i)\}} \geq m_k\left(\frac{b-1}{b}-\epsilon_k\right)
\]
is at least as big as the number of $s \in b^{m_k}$ such that 
$| \{ i <m_k\colon s(i)\neq 0\}| \geq m_k\left(\frac{b-1}{b}-\epsilon_k\right)$. 
From the law of large numbers we have that $\lim_{m_k\to \infty} \frac{e(m_k)}{b^{m_k}}=1$.
So, for large enough $m_k$ we have that for all $t \in b^{[0,n_k)}$
\begin{equation*}
\begin{split}
\frac{1}{b^{m_k}}\Bigg | \Bigg \{ s \in b^{m_k} & \colon \Big | \Big \{ i < m_k \colon 
E(t\conc s'\res [p_i-k, p_i-1]) \neq t\conc s'(p_i) \Big \} \Big | \geq m_k \left(\frac{b-1}{b}-\epsilon_k\right) \Bigg \} \Bigg | \\
& \geq 1-\frac{\epsilon_k}{l_k}
\end{split}
\end{equation*}
\noindent
It follows that for fixed $n_{k-1}$ that for all sufficiently large $n_k>n_{k-1}$
we have that for all $t \in b^{[0,n_{k-1}]}$ that 
\begin{equation*}
\begin{split}
\frac{1}{b^{m_k}} & \Bigg | \Bigg\{ s \in b^{m_k} \colon \forall E \in \sE_k\ 
\Big | \Big \{ i<n_k \colon E(t\conc s' \res [i-k,i-1])\neq t\conc  s'(i) \Big \} \Big | 
\\ &
\geq m_k\left(\frac{b-1}{b}-\epsilon_k\right) \Bigg \} \Bigg | 
\geq 1- l_k \frac{\epsilon_k}{l_k} =1-\epsilon_k.
\end{split}
\end{equation*}
\noindent
Since 
\begin{equation*}
\begin{split}
m_k\left(\frac{b-1}{b}-\epsilon_k\right)&  \geq [(d-\epsilon_k) n_k-n_{k-1}]\left(\frac{b-1}{b}-\epsilon_k\right)
\\ & 
=dn_k\left(\frac{b-1}{b}\right)\left[1 -\frac{\epsilon_k}{d}-\epsilon_k \frac{b-1}{b}- \frac{n_{k-1}}{dn_k}\right]
+n_k \epsilon_k^2 +n_{k-1}\epsilon_k
\\ & 
\geq \frac{d (b-1)n_k}{b}(1-3b \epsilon_k)
\end{split}
\end{equation*}
(assuming that $\frac{n_{k-1}}{n_k}< bd\epsilon_k$),
it follows that for all large enough $n_k$ and  all $t \in b^{[0,n_{k-1}]}$ we have that 
\begin{equation*}
\begin{split}
\frac{1}{b^{m_k}} & \Bigg | \Bigg \{s \in b^{m_k} \colon \forall E \in \sE_k\ 
\Big | \Big \{ i<n_k \colon E(t\conc s' \res [i-k,i-1])\neq t\conc  s'(i) \Big \} \Big | 
\\ &
\geq \frac{d (b-1)n_k}{b}(1-3b \epsilon_k) \Bigg \} \Bigg | 
\geq 1-\epsilon_k.
\end{split}
\end{equation*}
\noindent
We assume now that the $n_k$ are sufficiently fast growing to satisfy these inequalities. 
Since $\sum_k \epsilon_k<\infty$, it follows from Borel-Cantelli that 
for $\mu$ almost all $x \in b^\omega$ that for any $E\in \sE$, there are cofinitely many  $k$
such that 
\[
| \{ i<n_k \colon E(x' \res [i-k,i-1])\neq   x'(i) \}| 
\geq \frac{d (b-1)n_k}{b}(1-3b \epsilon_k) 
\]
\noindent
and thus for $\mu$ almost all $x \in b^\omega$ and all $E\in \sE$ we have that 
\[
\limsup_k \frac{1}{k} | \{ i< k \colon E(x' \res [i-k,i-1])\neq x'(i) \}| 
\geq \frac{d(b-1) }{b},
\]
\noindent
and so $\upe(x) \geq \frac{d(b-1)}{b}$.

\end{proof}

The next lemma suffices to show that $\norp(b)$ is $\bp^0_3$-complete.
We give an alternate, somewhat simpler, proof of the lemma after 
the current proof. However,  the first proof more resembles the 
proofs of the other parts of Theorem~\ref{mt} to follow.

\begin{lem} \label{a3}
For any $s \in [0,\frac{b-1}{b})$, $A_3(s)$ is $\bp^0_3$-hard.
\end{lem}

\begin{proof}
For $x \in 2^\omega$, we view $x$ as coding the sequence $x_0,x_1,\dots$ in $2^\omega$
where $x_i(j)=x(\langle i,j \rangle)$. We let $P=\{ x \colon \forall i\ \exists 
j_0\ \forall j \geq j_0\ x_i(j)=0\}$. It is well-known that $P$ is 
$\bp^0_3$-complete. 

We define a partition of $\omega$ into disjoint arithmetical sequences as follows. 
Let $I_0=\{ n\colon n \equiv 0 \mod 2\}$, be the set of even integers, and in general let 
$I_n=\{ n \colon n \equiv  2^n -1 \mod 2^{n+1}\}$. The $\{ I_n\}$ are pairwise disjoint 
arithmetic progressions, and $\omega=\bigcup_n I_n$. Note that 
$\omega-\bigcup_{k \leq m}I_k= \{ n\colon n \equiv 2^{m+1}-1 \mod 2^{m+1}\}$. 
Each $I_n$ clearly has density $\frac{1}{2^{n+1}}$.

Fix a set $J\subseteq \omega$ such that 
$\frac{b-1}{b}\left(1-\sum_{i \in J} \frac{1}{2^{i+1}}\right)=s$. 
Let $d_i=\frac{1}{2^{i+1}}$ be the density of $I_i$.
Let $B=\bigcup_{i \notin J}I_k$ and let 
$d_B=1-\sum_{i \in J}\frac{1}{2^{k+1}}=s\frac{b}{b-1}$ be the density of $B$.

We will take two fast growing sequences $\{ a_i\}_{i \in \omega}$
and $\{ b_j\}_{j \in \omega}$ of natural numbers (the $a_i$ will grow faster than the $b_j$).
We will then set $n_{i,j}=a_ib_j$. Also, let $B_{i,j}=[n_{i,j-1},n_{i,j})$ and 
$m_{i,j}=|B_{i,j} \cap I_i|$. Note that $\left|m_{i,j}-\frac{1}{2^{i+1}}(n_{i,j}-n_{i,j-1})\right|
\leq 1$. 

We first define the $b_j$.

Assume $b_0<b_1<\cdots<b_{j-1}$ have been defined. Let $b_j>b_{j-1}$
be large enough so that 
\begin{enumerate}
\item \label{t1}
For each $i\leq j$ with  $i\in J$, the density of $B\cup I_i$ in $[n_{j-1},n_j)$
is at least $(d_B+d_i)(1-\frac{1}{8j})$. 
\item  \label{t2}
$\frac{n_j-n_{j-1}}{n_{j-1}}  > \frac{1-\frac{1}{2j}}{1-\frac{1}{4j}}$.
\item \label{t3} 
For every $i \leq j$ and any $m \geq (d_B+d_i) (n_j-n_{j-1})(1-\frac{1}{8j})$, 
if $A \subseteq [n_{j-1},n_j)$ has size at least $m$, 
and $E\in \sE_j$, then we have that 
$\frac{p_m}{b^m} \geq 1-\frac{1}{j}\frac{1}{2^{j^2}}$, where 
$p_m$ is the number of $s \in b^{[n_{j-1},n_j)}$ such that 
$s$ is $0$ off of $A$ and we have that 
\begin{equation*}
\begin{split}
& |\{ k \in [n_{j-1},n_j) \colon E(s(k-j),\dots,s(k-1))\neq s(k) \}| 
\\ & \qquad \geq(d_B+ d_i)(n_j-n_{j-1}) \frac{b-1}{b} \left(1-\frac{1}{4j}\right).
\end{split}
\end{equation*}
\end{enumerate}
\noindent
Properties~(\ref{t1}) and (\ref{t2}) are easily satisfied, and property~(\ref{t3})
can be satisfied as in the proof of Lemma~\ref{tel}. Also, the properties continue to hold
if we replace $n_{j-1}$ and $n_j$ with $an_{j-1}$ and $an_j$ for any postive integer $a$.

We next define the sequence $a_0<a_1<\cdots$, which will be sufficiently fast growing
with respect to the $\{b_j\}$. Namely, such that for any $i$ and any $j>1$, 
$| i'\colon a_{i'} \leq a_ib_j\} | \leq j$. We could, for example, take 
$a_i=\prod_{i'\leq i}b_{i'}$.  
Set $n_{i,j}=a_ib_j$ and let $B_{i,j}=[n_{i,j-1},n_{i,j})$.
We define now the map $\psi \colon 2^\omega \to b^\omega$ which will be our
reduction from $P$ to $A_3(s)$. 
We first define a particular real $u \in 2^\omega$. To do this, consider 
integers $i <j$. 

\begin{claim} \label{tc}
$| \{ (i',j')\colon i' \geq i \wedge j'>0 \wedge B_{i',j'}\cap B_{i,j} \neq \emptyset\} |
\leq j^2$.
\end{claim}

\begin{proof}
For fixed $i' \geq i$, the number of $j'$ such that $B_{i',j'}\cap B_{i,j} \neq \emptyset$
has size at most $j$, since $B_{i,j}=[n_{i,j-1},n_{i,j})=[a_ib_{j-1},a_i,b_j)$,
 $B_{i',j+1}=[a_{i'}b_j,a_{i',b_{j+1}})$, and $a_{i'}\geq a_i$. Also, the number 
of $i'\geq i$ such that $B_{i',j'}\cap B_{i,j}\neq \emptyset$ for some $j'>0$
is at most $j$. This is because $n_{i',j'-1}=a_{i'}b_{j'-1}$, and if $n_{i',j'-1}\leq 
n_{i,j}$ then we must have $a_{i'} \leq \frac{a_ib_j}{b_{j'-1}}\leq a_ib_j$. 
From the choice of the $a_i$, this set has size at most $j$. 
\end{proof}

Fix for the moment any $j>i$, and consider the block $B_{i,j}$.
Let $W_{i,j} \subseteq \omega^2$ be the set of
$(i',j')$ with $i'\geq i$, $j'>0$, and with $B_{i',j'}\cap B_{i,j}\neq \emptyset$. 
By Claim~\ref{tc}, $|W_{i,j}|\leq j^2$. Consider the set $P_{i,j}$ of 
``patterns,'' by which we mean elements of $2^{W_{i,j}}$. For each pattern 
$p\in P_{i,j}$, and real $u \in b^\omega$, let $s(i,j,p,u)\in b^{B_{i,j}}$ 
be the sequence defined by:
\[
s(i,j,p,u)(k)=
\begin{cases}
u(k) & \text{ if } k \in B\\
u(k) & \text{ if } \exists i', j' \in W_{i,j}\ (P(i',j')=1 \wedge k \in (I_{i'}\cap B_{i',j'})) \\
0 & \text{ otherwise}
\end{cases}
\]

For each fixed pattern $p \in P_{i,j}$, 
from properties~\ref{t2}, \ref{t3} of the $b_j$ it follows that the $\mu$ measure of the 
set of $u \in b^\omega$ such that for all $t\in b^{[0,n_{j-1})}$, if $s'=t\conc s(i,j,p,u)$
then for all $E\in \sE_j$ we have
\begin{equation} \label{eqmi}
\begin{split}
& | \{ k \in B_{i,j}\colon E(s'(k-j),\dots,s'(k-1))\neq s'(k) \}| 
\geq \\ & \qquad (d_B+d_i) (n_j-n_{j-1})\frac{b-1}{b} \left( 1-\frac{1}{4j}\right) 
\end{split}
\end{equation}
\noindent
is at least $1-\frac{1}{j^2} \frac{1}{2^{j^2}}$. Since the number of
patterns in $P_{i,j}$ is at most $2^{j^2}$, we have that the $\mu$ 
measure of the set of $u \in b^\omega$ such that for all patterns $p \in P_{i,j}$
and all $E\in\sE_j$ equation~(\ref{eqmi}) holds is at least 
$1-\frac{1}{j^2}$. By Borel-Cantelli it follows that for fixed $i$ that for $\mu$ 
almost all $u \in b^\omega$ that there are cofinitely many $j$ such that 

($*_{i,j}$) For all $E\in \sE_j$,
and all $t \in b^{[0,n_{i,j-1})}$, and all patterns $p\in W_{i,j}$, if 
$s'=t\conc s(i,j,p,u)$ as above, then 
equation~(\ref{eqmi}) holds.

By countable additivity of $\mu$, for $\mu$ almost all
$u \in b^\omega$ the previous statement holds for all $i$. Fix $u\in b^\omega$ 
in this measure one set.

We now define the continuous map $\psi\colon 2^\omega \to b^\omega$ which will be the reduction from 
$P$ to $A_3(s)$. Let $x \in 2^\omega$ code the $x_i$, where $x_i(j)=x(\langle i,j\rangle)$. 
We define $\psi(x)$ by:

\[
\psi(x)(k)=\begin{cases}
u(k) & \text{ if } k \in B \\
u(k) & \text{ if }  \exists i,j\ 
(j>0 \wedge k \in (B_{i,j} \cap I_i) \wedge x_{i^{-}}(j)=1) \text{ where $i$ is the}\\
&  \quad \text{ $i^{-}$th element of $J$ }\\
0 & \text{ otherwise }
\end{cases}
\]
\noindent
The map $\psi$ is clearly continuous. We show that $x \in P$ iff $\upe(\psi(x))\leq s$. 
Suppose first that $x \in P$, so $\forall i\ \exists j_0\ \forall j \geq j_0\ x_i(j)=0$. 
Then for all $i \in J$ we have that for large enough $k \in I_i$ that 
$\psi(x)(k)=0$. The $E \in \sE$ (of width $1$) which constantly outputs $0$ 
will then  correctly compute $\psi(x)(k)$ for all sufficiently large 
$k \in I_i$ whenever $i \in J$. For $i \notin J$, we will have 
$\lim_{n\to \infty} \frac{1}{n} |\{ k <n \colon (k \in I_i) \wedge E(x \res [k-1,k-1])\neq x(k)\}|=
\frac{1}{2^{i+1}}\frac{b-1}{b}$. From this it follows that 
$\lim_{n \to \infty} \frac{1}{n}  |\{ k <n \colon E(x\res [k-1,k-1])\neq x(k)\}|=
\frac{b-1}{b}\sum_{i \notin J} \frac{1}{2^{i+1}}=s$. So, $\upe(\psi(x)) =s$
and so $\psi(x)\in A_3(s)$.

Suppose next that $x \notin P$. Let $i_0$ be least such that there are infinitely
many $j$ such that $x_{i_0}(j)=1$. Let $i^+_0$ be the $i_0$th element of $J$. 
If $i<i^+_0$ is in $J$, then for large enough $k\in I_i$ we have $\psi(x)(k)=0$.
If $i<i^+_0$ is  not in  $J$, then for large enough $k \in I_i$ we have $\psi(x)(k)=u(k)$. 
Also, for infinitely many $j$ we have that $\psi(x)(k)=u(k)$
for all $k \in I_{i^+_0} \cap B_{i^+_0,j}$. Fix $E\in \sE$. Say $E\in \sE_{j_0}$. 
From the definition of $u$, there are cofinitely many $j$ such that ($*_{i^+_0,j}$) holds. 
Intersecting a cofinite and an infinite set gives a $j>\max\{ i^+_0, j_0\}$ such that 
$x_{i_0}(j)=1$, ($*_{i^+_0,j}$) holds, and for all $i<i_0$  we have that 
$x_{i}(k)=0$ for all $k \geq j$. 
Thus, for all $i<i^+_0$ which are in $J$ we have that $u(k)=0$ 
for all $k \geq j$ with $k \in I_{i}$.
Let $p_0 \in P_{i^+_0,j}$ be the pattern such that 
$p_0(i',j')=1$ iff $i' \notin J$ or $[(i' \in J) \wedge x_{{i'}^-}(j')=1]$ 
where $i'$ is the ${i'}^-$th element of 
$J$. From ($*_{i^+_0,j}$) applied to $t=\psi(x)\res n_{i^+_0,j-1}$ and the pattern $p=p_0$, 
and noting that from  the definition of $\psi(x)$ that 
$\psi(x)\res B_{i^+_0,j}=\psi(x)\res [n_{i^+_0,j-1},n_{i,j})=s(i^+_0,j,p_0,u)$,
we have that

\begin{equation} \label{eqmj}
\begin{split}
& | \{ k \in B_{i^+_0,j}\colon E(\psi(x)(k-j_0),\dots,\psi(x)(k-1))\neq \psi(x)(k) \}| 
\\ & \qquad \geq
(d_B+d_i) (n_{i^+_0,j}-n_{i^+_0,j-1})\frac{b-1}{b} \left( 1-\frac{1}{4j}\right) 
\\ & \qquad = \left(s+d_i\frac{b-1}{b}\right) (n_{i^+_0,j}-n_{i^+_0,j-1}) \left( 1-\frac{1}{4j}\right),
\end{split}
\end{equation}
\noindent
and thus by property (\ref{t2}) of the $b_j$, the density of $k \in [0,n_{i^+_0,j})$ such that 
\[
E(\psi(x)(k-j_0),\dots,\psi(x)(k-1))\neq \psi(x)(k)
\]
is at least 
$\left(s+d_i\frac{b-1}{b}\right)  \left( 1-\frac{1}{2j}\right)$.
Since this holds for infinitely many $j$, we have that $\upe(\psi(x))\geq s+d_i\frac{b-1}{b}
>s$. Thus, $\psi(x)\notin A_3(s)$. 

\end{proof}
We now present the alternate simpler proof of Lemma~\ref{a3}.
\begin{proof}
Let $0 \leq s < \frac{b-1}{b}$. For each $i\in \N$ pick a probability vector 
$(p_{i,0}, \cdots, p_{i,b-1})$ such that $1 - \max_d p_{n,d} =  1 - p_{n,0} = s+ \frac{1- s-1/b}{n}$ 
and construct real numbers $u_n$ such that for any $k$
$$
\lim_v \frac{I(u_n,b_1 \cdots b_k,v)}{v} = \prod_{j=1}^k p_{n,b_j}.
$$
Then $\beta_{\ell}(u_i) = s + \frac{1-s-1/b}{i}$ for every $\ell$ and the 
constant $0$ block function $E_0$ is the minimizer of 
$\inf_{E\in \mathcal{E}_\ell} \limsup_N \beta_E(x,N)$.
Let $(a_j)$ be a sufficiently quickly growing sequence of integers with 
$a_1 = 1$ such that for $1 \leq \ell \leq j$, $1 \leq i \leq j$, and for all 
$n \geq \frac{a_j}{1-1/j}$ we have 
\begin{align*}
&\abs{\beta_{E_0}(u_i,n) - \beta_\ell(u_i)} = \abs{ \beta_{\ell}(u_i,n) - \beta_\ell(u_i)} < 
\frac{1}{j} \\ \intertext{and}
&\frac{a_{j+1}}{a_j} > j^2.
\end{align*}
Put $B_j = [a_j, a_{j+1})$.

We view $x \in 2^\omega$ as coding the sequence $x_0, x_1, \cdots$ in $2^\omega$ 
where $x_i(j) = x(\langle i,j\rangle)$. Define $m(j) = \min\{j, \min\{i : x_i(j) = 1 \} \}$ 
and define $\psi(x)$ by
$$
\psi(x)_{{}| B_j}(k)= u_{m(j)}(k - a_j).
$$
The map $\psi$ is continuous. Suppose $x \in P$. Then $\liminf_{j} m(j) = \infty$. Fix $\ell$ and note
\begin{align*}
& \inf_{E \in \mathcal{E}_\ell} \limsup_{N} \frac{1}{N} |\br{k \leq N : 
\psi(x)(k) \neq E(\psi(x)(k+1), \cdots, \psi(x)(k+\ell))}|\\
& = \limsup_N \frac{1}{N} | \br{k \leq N : \psi(x)(k) \neq 0} |\\
& = \limsup_j \sup_{N \in B_j} \sum_{k=1}^{j-1} \frac{1}{N} | \br{ v \leq | B_k| : 
u_{m(k)}(v) \neq 0}| + \frac{1}{N} | \br{ 0 \leq v \leq N-a_j : u_{m(j)}(v) \neq 0}|.
\end{align*}
Now for $N \geq a_j$ and $k \leq j-2$
\begin{align*}
\frac{1}{N} | \{1 \leq v \leq | B_k| : u_{m(k)}(v) \neq 0\} |
& = \frac{| B_k|}{N} \beta_{\ell}(u_{m(k)}, | B_k|) \\
& \leq \frac{a_{k+1} - a_k}{a_j} \leq \frac{a_{j-1}}{a_j} < \frac{1}{(j-1)^2}.
\end{align*}
and
\begin{align*}
\frac{1}{N} | \{1 \leq v \leq |B_{j-1}| : u_{m(j-1)}(v) \neq 0\}| 
& \leq \frac{a_{j}-a_{j-1}}{N} \pr{\beta_\ell(u_{m(j-1)}) + \frac{1}{j-1}} \\
& \leq \frac{a_j}{N} \pr{\beta_\ell(u_{m(j-1)}) + \frac{1}{j-1}}.
\end{align*}
For $a_j \leq N \leq \frac{a_j}{1-1/j} \leq a_{j+1} $ we have
\begin{align*}
\frac{1}{N} | \br{0 \leq v \leq N-a_j : u_{m(j)}(v) \neq 0}| \leq 
\frac{N-a_j}{N} \leq \frac{\frac{a_j}{1-1/j} - a_j}{a_j} = \frac{1}{j-1}.
\end{align*}
On the other hand for $\frac{a_j}{1-1/j} \leq N \leq a_{j+1}$
\begin{align*}
\frac{1}{N} | \br{0\leq v \leq N-a_j : u_{m(j)}(v) \neq 0}| 
&= \frac{N - a_j}{N} \beta_\ell(u_{m(j)},N-a_j)\\
&\leq \frac{N-a_j}{N} \pr{\beta_\ell(u_{m(j)}) + \frac{1}{j} }.
\end{align*}

Thus
\begin{align*}
\lim_\ell \beta^+_\ell(\psi(x)) &\leq \lim_\ell \limsup_j \sup \{t \beta_\ell(u_{m(j-1)}) 
+ (1-t) \beta_\ell(u_{m(j)}) : t \in [0,1]\}
\\ & \qquad
 + \frac{j-1}{(j-1)^2} + \frac{2}{j-1} + \frac{1}{j}\\
& = \lim_\ell \left( s + \limsup_j \max\br{\frac{1-s-1/b}{m(j-1)}, \frac{1-s-1/b}{m(j)}}\right) = s.
\end{align*}
Now suppose $x \notin P$. Then $\liminf_n m(n) = M < \infty$. Fix $\ell$ and note
\begin{align*}
&\inf_{E \in \mathcal{E}_\ell}  \limsup_{N} \frac{1}{N} |\br{k \leq N : 
\psi(x)(k) \neq E(\psi(x)(k+1), \cdots, \psi(x)(k+\ell))}|\\
& = \limsup_j \sup_{N \in B_j} \sum_{k=1}^{j-1} \frac{|B_k|}{N} 
\beta_\ell(u_{m(k)}, | B_k|) + \frac{N-a_j}{N} \beta_\ell(u_{m(j)}, N-a_j) \\
& \geq \limsup_j \frac{a_{j+1}-a_j}{a_{j+1}} \beta_\ell(u_{m(j)}, a_{j+1} - a_j)
\end{align*}
which implies
$$
\lim_\ell \beta^+_{\ell}(\psi(x)) \geq s + \frac{1-s-1/b}{M} > s.
$$
\end{proof}

We next show the lower bound for $A_2(s)$. The proof is similar to that for $A_3(s)$.

\begin{lem} \label{lema2}
For any $s \in (0,\frac{b-1}{b}]$ the set $A_2(s)$ is $\bp^0_3$-complete.
\end{lem}

\begin{proof}
We reduce the set $Q=\{ x \in 2^\omega \colon \forall i\ \exists j_0\ \forall j\geq j_0\
x_i(j)=1\}$ to $A_2(s)$. Let $J\subseteq \omega$ be such that 
$\frac{b-1}{b}\sum_{i\in J} d_i=s$, where we recall $d_i=\frac{1}{2^{i+1}}$ is the density of $I_i$. 
Let the sequences $\{ a_i\}$, $\{ b_j\}$ be as in Lemma~\ref{a3},
and as in that lemma let $n_{i,j}=a_i b_j$, $B_{i,j}=[n_{i,j-1},n_{i,j})$.
We use also the notion of pattern defined as follows. For $i<j$ let
\[
W_{i,j}=\{ (i',j')\colon j'>0 \wedge B_{i',j'}\cap B_{i,j}
\neq \emptyset \}
\]
\noindent
and let $P_{i,j}$ be the set of functions $p$ from $W_{i,j}$ to $\{0,1\}$ such that 
if $i'<i$ then $p(i',j_1)=p(i',j_2)$ for all $j_1, j_2$. Although $W_{i,j}$ now has size 
greater than $j^2$, this restriction on $p$ implies that $|P_{i,j}|
\leq 2^{j^2+i}\leq 2^{j^2+j}$. 
 
Arguing as in Lemma~\ref{a3}, for each $i \in \omega$ there is a $\mu$ measure one
set $A_i$ of $u \in b^\omega$ such that for cofinitely many $j$, all $E\in \sE_j$,
all $t \in b^{[0,n_{j-1})}$, 
and all patterns $p \in P_{i,j}$, if we define $s(i,j,p,u)\in b^{B_{i,j}}$ by

\[
s(i,j,p,u)(k)=
\begin{cases}
u(k) & \text{ if } \exists i', j' \in W_{i,j}\ (P(i',j')=1 \wedge k \in (I_{i'}\cap B_{i',j'})) \\
0 & \text{ otherwise}
\end{cases}
\]
\noindent
then if $s'=t \conc s(i,j,p,u)$ we have

\begin{equation} \label{eqmk}
\begin{split}
& | \{ k \in B_{i,j}\colon E(s'(k-j),\dots,s'(k-1))\neq s'(k) \}| 
\geq \\ & \qquad \pr{\sum_{\substack{i'\leq i \\  i'\in J}} d_i} 
(n_j-n_{j-1})\frac{b-1}{b} \left( 1-\frac{1}{4j}\right).
\end{split}
\end{equation}

We fix $u \in b^\omega$ in all of the measure one sets $A_i$. We define 
$\psi \colon 2^\omega \to b^\omega$ by:

\[
\psi(x)(k)=\begin{cases}
u(k) & \text{ if }  \exists i,j\ 
(i \in J \wedge j>0 \wedge k \in (B_{i,j} \cap I_i) \wedge x_{i^{-}}(j)=1) 
\text{ where }\\ & \quad 
\text{ $i$ is the $i^{-}$th element of $J$ }\\
0 & \text{ otherwise }
\end{cases}
\]

We show $\psi$ is a reduction from $Q$ to $A_2(s)$. 

If $x \notin Q$, then there is an $i_0$ such that for infinitely many $j$ 
we have $x_{i_0}(j)=0$. Let $i_0^+$ be the $i_0$th element of $J$. 
We may assume that $u$ is such that for every $i$ that 
$\lim_{n \to \infty} \frac{1}{n} |\{ k \in I_i \colon u(k)=0 \} | =\frac{d_i}{b}$. 
It follows that for $i\neq i_0^+$ and any $\epsilon>0$ that for all sufficiently large $j$ that 
\[
\frac{1}{|B_{i^+_0,j}|} |\{ k \in B_{i^+_0,j}\cap I_i\colon 
u(k)\neq 0\} | \leq (d_i+\epsilon) \frac{b-1}{b}.
\]
\noindent
It follows that for any $i_1>i^+_0$ and $\epsilon>0$ that for all sufficiently large $j$ with 
$x_{i_0}(j)=0$ that 
\[
\frac{1}{|B_{i^+_0,j}|} |\{ k \in B_{i^+_0,j}\colon 
\psi(x)(k)\neq 0\} | \leq \pr{\sum_{\substack{i\leq i_1\\ i\neq i^+_0}}d_i +\epsilon}\frac{b-1}{b}
+\sum_{i>i_1} d_i
\]
\noindent
It follows by considering the trivial $E\in \sE$ which always outputs $0$ that
\[
\loe(\psi(x)) \leq \sum_{\substack{i\neq i^+_0\\ i \in J}}d_i \frac{b-1}{b}<
\sum_{i\in J} d_i\frac{b-1}{b}=s.
\] 
Thus, $\psi(x) \notin A_2(s)$. 

Suppose next that $x \in Q$. Let $i \in J$. For all sufficiently large $j$ we
have that $\psi(x)(k)=u(k)$ for all $k \in B_{i,j} \cap I_{i'}$ where 
$i'\leq i$ and $i'\in J$. For such $j$, consider the pattern $p \in P_{i,j}$
such that $p(i',j')=1$ if $i'<i$, and for $i'\geq i$ with $i'\in J$ and $j'>0$ we have 
$p(i',j')=x_{{i'}^-}(j')$, where $i'$ is the ${i'}^-$th element of $J$. From the properties of
$u$ we have that for all $E \in \sE_j$ that 
\begin{equation*}
\begin{split}
&\frac{1}{|B_{i,j}|} |\{ k \in B_{i,j} \colon 
E(\psi(x)(k-j),\dots,\psi(x)(k-1))\neq \psi(x)(k) \} |
\\ & \quad \geq \pr{\sum_{\substack{i' \leq i \\ i' \in J}} d_i }\pr{\frac{b-1}{b}} \left( 1- \frac{1}{4j} \right).
\end{split}
\end{equation*}
\noindent
This gives that $\loe(\psi(x)) \geq (\sum_{i \in J} d_i) (\frac{b-1}{b})=s$, and 
so $\psi(x) \in A_2(s)$.
\end{proof}

We next show that the set $A_4(s)$ is $\bp^0_2$-complete for any $s \in (0,\frac{b-1}{b}]$.

\begin{lem} \label{lema4}
For any $s \in (0,\frac{b-1}{b}]$, the set $A_4(s)$ is $\bp^0_2$-complete. 
\end{lem}

\begin{proof}
Let $Q \subseteq 2^\omega$ be the set $Q=\{ x \colon \forall i\ \exists j \geq i\ 
x(j)=1\}$. Let $A \subseteq \omega$ be a set of density $d=s \frac{b}{b-1}\in (0,1]$. 
We define the integers $b_0<b_1<\cdots$ inductively, and set 
$B_j=[b_{j-1},b_j]$. Given $b_0,\dots,b_{j-1}$, we let $b_j>b_{j-1}$
be large enough so that for each $t \in 2^{b_{j-1}}$, we have that for all $E\in \sE_j$ 
that 
\begin{equation*}
\begin{split}
& \Bigg | \Bigg \{ u \in 2^{B_j} \colon \frac{1}{|B_j|} \Big | \Big \{ k \in B_j \colon E(t\conc u(k-j),
\dots, t\conc u(k-1)) \neq t\conc u(k) \Big \} \Big | \\ & \quad \geq \frac{b-1}{b} d-\frac{1}{j} \Bigg \} \Bigg |
\geq \left( 1-\frac{1}{2^j j^2} \right) 2^{|B_j|}.
\end{split}
\end{equation*}
\noindent
This is possible by law of large numbers and Lemma~\ref{tel}.
By Borel-Cantelli, for $\mu$ almost all $x \in 2^\omega$ we have that 
\[
\exists j_0\ \forall j \geq j_0\ \forall t \in 2^{B_j}\ \forall E \in \sE_j\ 
\frac{1}{|B_j|}| \{ k \in B_j \colon E(x_s(k-j),
\dots, x_t(k-1)) \neq x_t(k) \}| \geq s-\frac{1}{j}.
\]
\noindent
where $x_t$ is the result of replacing $x\res [0,b_{j-1})$ with $t$. 
Fix $u\in 2^\omega$ in this measure one set. 

We define $\psi \colon 2^\omega \to 2^\omega$ by:

\[
\psi(x)(k)=
\begin{cases}
u(k) &\text{ if } \exists j\ (k\in B_j \wedge x(j)=1)\\
0 & \text{ otherwise }
\end{cases}
\]
If $x \in Q$, then there are infinitely many $j$ such that $x(j)=1$ and thus 
$\psi(x)\res B_j=u\res B_j$. From the definition of $u$, there is a tail of these $j$ for which 
\[
\forall E \in \sE_j\ 
\frac{1}{|B_j|}| \{ k \in B_j \colon E(\psi(x)(k-j),
\dots, \psi(x)(k-1)) \neq \psi(x)(k) \}| \geq s-\frac{1}{j}.
\]
Thus, for any $E \in \sE$ we have that 
\[
\limsup_j 
\frac{1}{|B_j|}| \{ k \in B_j \colon E(\psi(x)(k-j),
\dots, \psi(x)(k-1)) \neq \psi(x)(k) \}| \geq s.
\]
We may assume that 
$\frac{\sum_{i<j}|B_i|}{|B_j|} \to 0$ with $j$, and it follows that 
$\upe (\psi(x))\geq s$, so $\psi(x) \in A_4(s)$. 

If $x \notin Q$, then $\psi(x)(k)$ is $0$ for all large enough $k$. This gives that 
$\upe(\psi(x))=0 <s$, and so $\psi(x) \notin A_4(s)$.
\end{proof}

We next show the lower-bound for $A_1(s)$.

\begin{lem} \label{inflower}
For $s \in [0,\frac{b-1}{b})$, the set $A_1(s)$ is $\bp^0_4$-hard.
\end{lem}

\begin{proof}
Fix $s \in [0,\frac{b-1}{b})$. Let $R=\{ x \in (2^\omega)^3 \colon 
\forall i\ \exists j_0\ \forall j \geq j_0\ \exists k\ x(i,j,k)=0\}$. 
$R$ is a $\bp^0_4$-complete set, and so it suffices to reduce $R$ to 
$A_1(s)$. Let $J\subseteq \omega$ be such that $s=\frac{b-1}{b} \left(1-\sum_{i \in J}\frac{1}{2^{i+1}}\right)$.
Let the $I_i$ partition $\omega$ as before, so $d(I_i)=\frac{1}{2^{i+1}}$.

We will define a fast growing sequence $b_0<b_1<\cdots$,
and we will also let $B_n=[b_n,b_{n+1})$. Each $n$
codes a triple $n=\langle i_n,j_n,t_n\rangle$, where $i_n,j_nt_n\leq n$. 
We will also define a certain sufficiently fast growing function $g
\colon \omega \to \omega$ (this will be the map $j \mapsto p(j,\frac{1}{j})$ 
of Claim~\ref{u2}). 
Also as in 
the proof of Lemma~\ref{a3} we will fix a particular $u \in 2^\omega$ from a 
certain $\mu$ measure one set which will guide the construction of the reduction map $\psi$. 
The construction will be similar to that of Lemma~\ref{a3}, the main difference being that 
at some points of the construction instead of copying $0$s to parts of the block 
$B_n$ we will copy a portion of $u$ repeated with period $g(j_n)$.

\begin{claim} \label{u1}
Let $\omega=A\cup B \cup C$, a disjoint union, and assume $A$, $B$, $C$ have period 
$p$ (that is, $\chi_A$, $\chi_B$, $\chi_C$  have period $p$). Let $A$, $B$, $C$ 
have densities $d_A$, $d_B$, $d_C$ respectively. Then for $\mu$ almost all $u \in 2^\omega$
(in fact, if $u$ is normal in base $b$)
we have the following. There  is an $E\in \sE$ such that 
for any $\bu$ with $\bu\res A= u\res A$ and $\bu\res B$ of period $p$, we have that 
$\upe_E(\bu)\leq d_A \frac{b-1}{b} +d_C$.
\end{claim}

\begin{proof}
Let $u$ be normal in base $b$. Let $\epsilon >0$. 
Consider sequences $s \in b^{np}$ for some integer $n$. 
For large enough $n$, the probability that a sequence $s' \in b^{A\cap np}$ will have the property that 
$A'=\{ i \in A \cap np\colon s'(i)=s'(i-p)\}$ has size at least 
$\frac{d_Apn}{b}(1-\frac{\epsilon}{3})$ is at least $1-\frac{\epsilon}{3}$. 
This follows by the argument of Lemma~\ref{tel}. We call such an $s'$ good. 
Since $u$ is normal in base $b^{d_Apn}$, it follows that for large enough $N$ 
that the number of $k \leq N$ such that $u \res (A\cap [kpn,(k+1)pn) )$ is good 
is at least $N(1-\frac{\epsilon}{2})$. For such $N$ we have that 
\begin{equation*}
\begin{split}
| \{ i \in A\cap d_A p n N & \colon u(i)=u(i-p)\}|  \geq 
N\left(1-\frac{\epsilon}{2}\right)\left(\frac{d_Apn}{b}\right)\left(1-\frac{\epsilon}{3}\right) 
\\ & \geq  \left(\frac{d_ApnN}{b}\right) (1-\epsilon). 
\end{split}
\end{equation*}

\noindent
Let $E\in \sE_{p+1}$ be the block
function such that $E(s)=s(0)$, so $E$ is simply guessing that $x(n)$
will be $x(n-p)$. It follows that if $\bu \in b^\omega$ 
is such that $\bu \res A$ is normal, then for large enough $N$ we have that 
\begin{equation*}
\begin{split}
| \{ i \leq pnN  \colon \bu(i) = E(\bu(i-p),\dots, \bu(i-1)) \} |
& \geq \left(\frac{d_ApnN}{b}\right) (1-\epsilon) + d_B pn N \\ & =pnN \left(\frac{d_A}{b}(1-\epsilon)+d_B\right).
\end{split}
\end{equation*}

\noindent
Thus, $\upe_E(\bu)\leq 1-\left(\frac{d_A}{b}(1-\epsilon)+d_B\right)= d_A\frac{b-1}{b}+d_C +\epsilon \frac{D_A}{b}$.
Since $\epsilon$ was arbitrary, $\upe_E(\bu)\leq d_A \frac{b-1}{b} +d_C$.


\end{proof}

For $u \in b^\omega$ and $p,q\in \omega$, we define $\bu(k,p)$ by $\bu(k,p)(m)=
u(m)$ for $m <k$, and for $m \geq k$ we set $\bu(k,p)(m)=u(m-k \mod p)$. 
Thus, after the first $k$ digits of $u$, we repeat the digits of $u$ periodically
with period $p$.

\begin{claim} \label{u2}
Let $j_0\in \omega$, $\epsilon,\epsilon' > 0$. Then there is a  $p=p(j_0,\epsilon) \in \omega$ which
is a power of $2$ and an $\eta=\eta(p,\epsilon')$ such that if 
$A, B, C \subseteq \omega$  are disjoint and of period $p$ with densities $d_A$, 
$d_B$, $d_C$ respectively and with $\omega=A\cup B \cup C$, then
for any $n=p\ell \geq \eta$, 
we have that $\frac{|H|}{b^{(A\cup B)\cap [0,n]}} 
\geq 1-\epsilon'$,
where $H$ is the set of $u \in b^n$ such that for any $u_c \in b^{C\cap [0, n]}$ 
if $u' \in b^n$ is defined by 
\[
u'(i) =\begin{cases}
u(i) & \text{ if } i \in A \\
\bu(0,p) & \text{ if } i \in B \\
u_c(i) & \text{ if } i \in C
\end{cases}
\]
then for any $E\in \sE_{j_0}$ we have 
\begin{equation*} 
\begin{split}
(*)\qquad  
\forall k \in [\eta,n]\ \, \Big [ \frac{1}{k}
|\{ i \leq k & \colon E(u'(i-j_0),\dots,u'(i-1)) \neq u'(i)\}| \\ & 
\geq (d_A+d_B)\frac{b-1}{b}(1 -\epsilon)-j_0 d_C \Big ].
\end{split}
\end{equation*}
Furthermore, this holds for all $p' \geq p$.
\end{claim}

\begin{proof}
Fix $j_0$, $E\in \sE_{j_0}$, $\epsilon >0$. We show that for large enough 
$p \in \omega$, and $\omega=A \cup B \cup C$, a disjoint union with $A,B,C$ 
having period $p$, and $u_c\in b^C$, that the probability a $u \in b^{(A\cup B)\cap [0,n]}$
satisfies $(*)$ is as close to $1$ as desired. 

Given $u \in b^{(A\cup B)\cap [0,n]}$, let $u''\in b^{[0,n]}$ be defined as $u'$ above
except we put $u''(i)=0$ for $i\in C \cap [0,n]$. 
The argument of Lemma~\ref{tel} shows that as $p$ grows, with probability 
approaching $1$ in choosing $u \res ((A\cup B) \cap [0,p])$ we have that 
\[
(**) \qquad \frac{1}{p} | \{ i \leq p \colon E(u''(i-j_0),\dots,u''(i-1)) \neq u''(i)\} |
\geq (d_A+d_B)\frac{b-1}{b}\left(1-\frac{\epsilon}{2}\right).
\]
From the finite version of Fubini it follows that for any $\delta>0$ that for large 
enough $p$ that with probability at least $1-\delta$ in chosing $u_b=u\res (B\cap [0,p])$
that for probability at least $1-\delta$ in choosing $u_a=u\res (A\cap [0,p])$ we have
that if $u''$ is defined from $u_a$ and $u_b$, then $(**)$ holds. 

Choose $u_b$ in this set of measure at least $1-\delta$. Let $G$ be the set of $u_a=u\res (A\cap [0,p])$
such that if $u''=u_a \cup u_b\cup u_c$ (with $u_c=0$) then $u$ satisfies $(**)$. 
So, $G$ has measure at least $1-\delta$. 
From the law of large numbers we have that the probability that a $u \in b^\omega$ 
has the property that for all $k\geq 1$ 

\[
| \{k' \leq k \colon u \res (A\cap [k'p ,(k'+1)p)) \in G\} | \geq k\left(1-\frac{\epsilon}{2}\right)
\]
is at least $1-h(\delta)$, where $h(\delta)\to 0$ as $\delta \to 0$. 
We choose $\delta< \frac{\epsilon}{2}$ small enough so that 
$h(\delta)<  \frac{\epsilon}{2}$. It follows that for all $n=\ell p \geq p$ that with probability 
at least $1-\delta$, $u_b \in b^{B\cap [0,p)}$ has the property that with 
probability at least $1-h(\delta)$, $u_A=u\res (A\cap [0,n)$ has the property that if 
$u''\in b^{[0,n)}$ is formed from $u_A$, $u_b$, $u_c=0$ as above, then 
for all $k'\leq \ell$

\begin{equation} \label{equ}
\begin{split}
| \{ i \leq k'p\colon & E(u''(i-j_0),\dots,u''(i-1)) \neq u''(i)\}| \\ & \geq 
k'p \left(1-\frac{\epsilon}{2}\right)(d_A+d_B)\left( \frac{b-1}{b}\right) \left( 1-\frac{\epsilon}{2}\right)
\\ & \geq k'p (d_A+d_B)\left( \frac{b-1}{b}\right) (1-\epsilon).
\end{split}
\end{equation}

Since $\delta, h(\delta)<\frac{\epsilon}{2}$ it follows that 
for any $n=\ell p \geq p$ that with probability at least $1-\epsilon$, 
$u \in [0,n)$ has the property that if $u''$ is obtained from $u$ and $u_c$ as above,
then for all $k'\leq \ell$ we have that Equation~(\ref{equ}) holds.
If we set $\eta(p,\epsilon)=\frac{p}{\epsilon}$, then we get the inequality of 
Equation~\ref{equ} replacing $i\leq k'p$ with $i\leq k$ for any 
$k \in [\eta,n]$.

Consider then $u' \in b^{[0,n)}$, which is defined as $u''$ except we use 
the given $u_c$ instead of the $0$ function. We have that 
$E(u''(i-j_0),\dots,u''(i-1))$ and $E(u'(i-j_0),\dots,u'(i-1))$ can only disagree 
if $C\cap [i-j_0,i) \neq \emptyset$. Thus, there can be at most 
$d_C n j_0$ many $i \in [0,n)$ where such a disagreement occurs. The inequality 
of the claim then follows.
\end{proof}

\begin{claim} \label{u2b}
For every $m, \epsilon, \epsilon'$ there is an 
$\eta=\eta(m,\epsilon,\epsilon')$ such that for any $i_0, j_0, m_0\leq m$,  
and any $p \geq 1$, for all large enough $n$ we have that 
if $\ds A=\bigcup_{\substack{i \notin J\\ i \leq m_0}}I_i \cup 
\bigcup_{\substack{i > i_0\\ i \leq m_0}} I_i$, 
$\ds B=\bigcup_{\substack{i\leq i_0\\ i \in J}}I_i$, 
and $C=\omega-(A\cup B)$, 
then $\frac{|H|}{b^n} \geq (1-\epsilon')$,
where $H$ is the set of $u \in b^n$ such that if 
\[
u'(i) =\begin{cases}
u(i) & \text{ if } i \in A\cup C \\
\bu(0,p) & \text{ if } i \in B 
\end{cases}
\]
then for any $E\in \sE_{j_0}$ we have 
\begin{equation*} \label{equ2b}
\begin{split}
\forall k \in [\eta,n]\ \, \Big [ \frac{1}{k}
|\{ i \leq k & \colon E(u'(i-j_0),\dots,u'(i-1)) \neq u'(i)\}| \\ & 
\geq d_A\left(\frac{b-1}{b}\right)(1 -\epsilon) \Big ].
\end{split}
\end{equation*}
\end{claim}

\begin{proof}
It is enough to fix $i_0$, $j_0$, $m_0$, $\epsilon$, $\epsilon'$, and $p$ and we show that for large 
for large enough $n$ the stated property holds. 
Note that there are 
$\leq b^p$ many choice for $u_B=u\res (B\cap [0,n))$. It is enough to fix
a choice for $u_B$ are show that for large enough $n$ the property holds. 
This, however, follows immediately from the argument of Lemma~\ref{tel}.

\end{proof}

From Claim~\ref{u2} we have the following. Let $b_0<b_1<\cdots <b_n <\cdots$ 
be a sufficiently fast growing sequence of powers of $2$ (exactly how fast the 
sequence needs to grow will be specified below).

\begin{claim} \label{u3}
For almost all $u \in b^\omega$ we have the following. 
For any $j_0$ and $E\in \sE_{j_0}$, 
for all large enough $n$ we have that 
if $[b_n,b_{n+1})=A\cup B$, a disjoint union
of sets of period $\leq 2^n$,  
where 
$\ds A=[b_n, b_{n+1}) \cap (\bigcup_{i \notin J} I_i \cup \bigcup_{i > i_0} I_i)$ and 
$\ds B=[b_n\cap b_{n+1}) \cap (\bigcup_{\substack{i\leq i_0\\ i \in J}}I_i)$ for some $i_0\leq n$, 
and if $j \leq n$ and $p=p(j,\frac{1}{j})$, $\eta\geq \eta(p,\frac{1}{n^4})$ 
as in claim~\ref{u2}, 
$\eta \geq \eta(n,\frac{1}{j},\frac{1}{n^4})$ as in Claim~\ref{u2b}, 
then if $u'$ is defined as in Claim~\ref{u2b}, then
we have:

(1) (large $j$ case) 
If $j \geq j_0$ then 
\begin{equation*}
\begin{split}
\forall k \in [b_n+\eta,b_{n+1}]\ \, \Big [
\frac{1}{k-b_n} \, | \{ i \in [b_n,k) & \colon E(u'(i-j),\dots,u'(i-1)) \neq u'(i)\} |
\\ & \geq  (d_A+d_B)\left(\frac{b-1}{b}\right)\left(1-\frac{1}{j}\right) - \frac{j}{2^n} \Big ].
\end{split}
\end{equation*}

(2) (general case) 
\begin{equation*}
\begin{split}
\forall k \in [b_n+\eta,b_{n+1}] \ \, \Big [ \frac{1}{k-b_n}
|\{ i \in [b_n,k) & \colon E(u'(i-j),\dots,u'(i-1)) \neq u'(i)\} |
\\ & \geq d_A\left(\frac{b-1}{b}\right)\left(1-\frac{1}{n}\right) \Big ].
\end{split}
\end{equation*}
where $d_A$, $d_B$ are the densities of $A$, $B$ in $[b_n,b_{n+1})$.
\end{claim}

\begin{proof}
It is enough to fix $j_0$, $E \in \sE_{j_0}$, and show that almost all $u$
have the desired property for these values. 
By Borel-Cantelli it is enough to show that the probability that $u \res [b_b,b_{n+1})$
satisfies the conclusion of the claim it at least $1-\frac{1}{n^2}$. Fix $b_n$,
and we show, assuming $b_{n+1}$ is sufficiently large compared to $b_n$, that 
the interval $[b_b,b_{n+1})$ has this property. 
There are at most $n$ many partitions $[b_n,b_{n+1})=A\cup B$ of the type 
stated in the claim (since the choice of $A,B$ is determined by $i_0\leq n$), 
so it enough to fix $A,B$ and show that with probability at least 
$1-\frac{1}{n^3}$ in choosing $u\res [b_n,b_{n+1})$ the statement of the claim holds. 
Note that the $j$ in the claim satisfy $j\leq n$, so there is a bound $\eta_n$ 
depending only on $n$, such that if $j \leq n$, $p=p(j,\frac{1}{j})$, then 
$\eta(p, \frac{1}{n^3})\leq \eta_n$. We will assume that $b_n \gg \eta_n$ for all $n$,
and in particular we will choose $b_{n+1}$ so that $\frac{\eta_{n+1}}{b_{n+1}-b_n} <\frac{1}{n}$. 
This is possible as $\eta_{n+1}$ depends only on $n+1$ and not the value of $b_{n+1}$. 
Similarly we may fix $j \leq n$ and show that with probability at least 
$1-\frac{1}{n^4}$ in choosing $u \res [b_n,b_{n+1})$ we satisfy the claim. 
In case (1), that is, $j\geq j_0$, the conclusion follows immediately from 
Claim~\ref{u2}, assuming that $b_{n+1}> b_n+\eta_n$. 
In applying Claim~\ref{u2} we use $C=\bigcup_{i>n}I_i$.
In case (2),
the conclusion follows immediately from Claim~\ref{u2b}, assuming again that 
$b_{n+1}$ is sufficiently large compared to $b_n$ (the interval $[0,n)$
of Claim~\ref{u2b} becomes $[b_n,b_{n+1})$ here). 

\end{proof}

We now fix $u \in 2^\omega$ in the measure one set described in Claim~\ref{u3}.
We also fix the fast growing sequence $b_0<b_1<\cdots <b_n<b_{n+1}<\cdots$. 
As we said in Claim~\ref{u3}, we take $b_{n+1}>b_n+\eta_n$,
where $\eta_n$ is the maximum of $\eta(p,\frac{1}{n^4})$, where $p=p(n,\frac{1}{n})$,
from Claim~\ref{u2} and $\eta(n, \frac{1}{n},\frac{1}{n^4})$ from Claim~\ref{u2b}.

Claim~\ref{u3} then gives the following property of $u$ and the $b_n$. 
\medskip

$(\dagger)$: For any $j_0$, $E\in \sE_{j_0}$,
for all large enough $n$ if $[b_n,b_{n+1})=A\cup B$ where 
$A=[b_n,b_{n+1}) \cap(\bigcup_{i \notin J} I_i \bigcup_{\substack{i\in J\\ i >i_0}} I_i)$,
$B=[b_n,b_{n+1})\cap (\bigcup_{\substack{i \in J\\ i \leq i_0}} I_i)$ where $i_0 \leq n$, 
then if $p=p(j,\frac{1}{j})$ (as in Claim~\ref{u2}) and $u'$
is defined from $u$ and $p$ as in Claim~\ref{u2b}, then for any $j \leq n$  
if we set 
\[
d=\frac{1}{k-b_n}
|\{ i \in [b_n,k)  \colon E(u'(i-j),\dots,u'(i-1)) \neq u'(i)\} |
\]
\noindent
then for $b_n+\eta_n\leq k \leq b_{n+1}$ we have:
\begin{enumerate}
\item
If $j \geq j_0$ then $d\geq (d_A+d_B)\left(\frac{b-1}{b}\right)\left(1-\frac{1}{j}\right)- \frac{j}{2^n}$.
\item
$d\geq d_A\left(\frac{b-1}{b}\right)\left(1-\frac{1}{n}\right)$.
\end{enumerate}
\medskip

We now return to the proof of Lemma~\ref{inflower}.

Given $x \in 2^{\omega^3}$, we define a function $h(x)\colon \omega^3 \to \omega^3$ as follows. 
We let $h(x)(i,j,0)$ be $(1,j+1,0)$ if $x(i',j,0)=0$ for all $i' \leq i$. 
Otherwise, set $h(x)(i,j,0)=(0,j,1)$. In general, we define 
\begin{equation*}
h(x)(i,j,t)=
\begin{cases}
(1,  (h(x)(i,j,t-1)))_1+1,0) & \text{ if } \\ 
{} \qquad   {} 
\forall i'\leq i\ \exists t' \leq t\ x(i',(h(x)(i,j,t-1))_1,t'))_1=0 & \\ 
(0, (h(x)(i,j,t-1))_1,(h(x)(i,j,t-1))_2+1) & \text{ otherwise}
\end{cases}
\end{equation*}

The function $h$ does the following. The input $i$ sets the ``width'' of the search,
that is, it will search over the $(x)_{i'}$ for $i' \leq i$. The input $j$ 
sets the initial start of the search in that the search will begin at the 
$x(i',j,0)$. The search checks to see if all of these are equal to $0$. If so,
it will output $(h(x)(i,j,0))_0=1$, denoting a successful search, and then 
replace $j$ with $j+1$ and begin a new search at $x(i,j+1,0)$. 
The output $h(x)(i,j,0))_1$ records the new value $j+1$, and the 
output $h(x)(i,j,0))_2$ records the new ``height'' of the search, will in this case
is set back to $0$. 
If not all the values $x(i',j,0)$ are $0$, then the $j$ value remains the same and
we increment the height of the search. The means we will search the values 
$x(i',j,k')$ where $i'\leq i$, $j$ is the current value of $(h(x)(i,j,t))_1$,
and $k'\leq (h(x)(i,j,t))_2$, which is the current height. For a given $j$,
we keep incrementing the height $k$ and see if 
$\forall i'\leq i\ \exists k'\leq k\  x(i',j,k')=0$. If so, the search is successful,
and we then increment $j$ to $j+1$ and reset the height $k$ to $0$. Otherwise,
we continue to increment the height $k$ and continue the search.

The search is attempting to verify,
step by step, that 
\[
\forall i'\leq i\ \forall j' \geq j\ \exists k'\ x(i',j',k')=0.
\]
If this condition holds for some $i,j$, then $(h(x)(i,j,t))_1$ will tend to $\infty$ with $t$. 
If this condition fails for $i,j$, then for $j'$ the least integer $\geq j$ 
such that $\neg \forall i' \leq i\ \exists k\ x(i',j',k)=0$
we have that $(h(x)(i,j,t))_1$ will be equal to $j'$ for all large enough $t$
(the search will ``get stuck'' at $j'$).

Recall that $\{ b_n\}$ is a sufficiently fast growing sequence, so that $u$ 
and the $\{ b_n\}$ satisfy Claim~\ref{u3}. 
As before, we let $B_n=[b_n,b_{n+1})$.
We view $n$ as coding a triple of integers which we write as $(i_n,j_n,t_n)$. 
We define the map $\psi \colon 2^{\omega^3} \to 2^\omega$ as follows. Let $x \in 2^{\omega^3}$. 
Recall the $I_i$ are the pairwise disjoint arithmetical sequences of Lemma~\ref{a3},
so $I_i$ has density $\frac{1}{2^{i+1}}$. Also,  $J \subseteq \omega$ is 
such that  $\frac{b-1}{b}(1-\sum_{i \in J} \frac{1}{2^{i+1}})=s$. For $i \notin J$, we will
just copy $u$ to $I_i$. It remains to specify $\psi(x)(m)$ for $m \in I_i$ where $i\in J$.

For $m \in B_n \cap I_i$, where $i \in J$, we define $\psi(x)(m)$ through the following cases
(the definition of $\bu(k,p)$ is given right before Claim~\ref{u2}).

\begin{enumerate}
\item
If $i \notin J$, we set $\psi(x)(m)=u(m)$.
\item
If $i \in J$ and $m \in B_n$, we set $\psi(x)(m)=u(m)$ if $i_n <i$. 
\item
If $i\in J$, $m \in B_n$, and $i_n \geq i$, then if $(h(x)(i_n,j_n,t_n))_0
=0$ (unsuccessful search at step $t_n$ for $(i_n,j_n)$) we set $\psi(x)(m)=u(m)$.
\item
If $i\in J$, $m \in B_n$, $i_n \geq i$, and $(h(x)(i_n,j_n,t_n))_0 =1$ (successful search
at step $t_n$), we set $\psi(x)(m)=\bu (b_n,p(j_n,\frac{1}{j_n}))(m)$, where 
$p(j_n,\frac{1}{j_n})$ is defined in Claim~\ref{u2}. 
\end{enumerate}

We show that $\psi$ is a reduction from the $\bp^0_4$ set $R$ to the set 
\[
A_1(s)=\{ z \in b^\omega \colon \loe(z) \leq s\}.
\]

First assume that $x \in R$. Fix $\epsilon >0$. Let $i_0$ be large enough so that 
\[
\displaystyle \frac{b-1}{b} (\sum_{\substack{i \notin J\\ i\leq i_0}} d_n)+
\sum_{i >i_0}d_n <s+\frac{\epsilon}{2},
\]

\noindent
where $d_i=\frac{1}{2^{i+1}}$ is the density of $I_i$. 
Let $j_0$ be large enough so that $\forall i\leq i_0\ \forall j \geq j_0\ 
\exists k\ x(i,j,k)=0$. We can do this as $x \in R$. 
Let $A \subseteq \omega$
be given by 
\[
n \in A \Leftrightarrow (i_n=i_0 \wedge j_n=j_0) \wedge 
( (h(x)(i_n,j_n,t_n))_0=1),
\]
that is, the $(i_n,j_n)$ search at step $t_n$ is successful.
From the definition of $j_0$ and the properties of $h(x)$ 
we have that $A$ is infinite (that is, there are infinitely many $t$ such that 
$(h(x)(i_0,j_0,t))_0=1$). For any $n \in A$, $i \leq i_0$ in $J$, and 
$m \in B_n \cap I_i$, we have that $\psi(x)(m)=\bu(b_n, p(j_0, \frac{1}{j_0}))$.
It follows from Claim~\ref{u1}, and the fact that the $b_n$ grow sufficiently fast, 
that there is an $E\in \sE$ (with say $E\in \sE_r$) such that for all 
large enough $n \in A$ we have that 
\begin{equation*}
\begin{split}
& \frac{1}{|B_n|} 
|\{ m \in B_n \colon 
E(\psi(x)(m-r),\dots,\psi(x)(m-1))\neq \psi(x)(m) \} |
\\ & \qquad\leq \frac{b-1}{b}(1-\sum_{\substack{i\in J\\ i \leq i_0}} d_n)+
\frac{\epsilon}{2} <s+\epsilon.
\end{split}
\end{equation*}

\noindent
Since $\epsilon>0$ was arbitrary, we have that $\loe(\psi(x)) \leq s$, that is,
$\psi(x)\in A_1$.

Assume next that $x \notin R$. Let $i_0$ be least so that $\forall j\ \exists j' \geq j\ 
\forall k\ x(i,j',k)=1$. So, for any $j$, $(h(x)(i_0,j,t))_1$ has a limiting value $j' \geq j$
as $t$ goes to infinity (i.e., the width $i_0$ search starting at $j$ will always gets stuck). 
Note that if $i_1 >i_0$, then $(h(x)(i_1,j,t))_1$ will reach its limiting value 
at or before when $(h(x)(i_0,j,t))_1$ does, that is, the $h(x)(i_1,j,t)$ search will 
get stuck at or before when $h(x)(i_0,j,t)$ gets stuck. So, for all $j$ we have that for all
sufficiently large $n$ with $i_n \geq i_0$, and $j_n=j$, that 
$(h(x)(i_n,j_n,t_n))_0=0$ (unsuccessful search at step $t_n$).

Consider $\psi(x)\res I_i$ where $i\geq i_0$.
So, for each $j$ we have that for all large enough $n$ with $j_n=j$ 
that either $i_n < i$, in which case $\psi(x)\res (B_n \cap I_i)=
u\res (B_n \cap I_{i})$, or else $i_n \geq i$, in which case  (since 
$(h(i_n,j,t_n))_0=0$) we also have that 
$\psi(x)\res (B_n \cap I_{i})=u\res (B_n \cap I_{i})$. 
Thus, for any $j$ we have that for large enough $n$ that 
$\psi(x)\res (B_n \cap I_{i})$ is either of the form $u\res (B_n \cap I_{i})$
or else of the form $\bu(b_n,p)$ where $p >j$. That is, the periods $p$ 
used in truncating $u$ in the block $B_n$ go to infinity with $n$.

It follows from $(\dagger)$ (where the $i_0$ there is the current $i_n$) 
that for any $\epsilon >0$ and any $E\in \sE_{j_0}$ that for all
large enough $n$ that in either case, $i_n <i_0$ or $i_n >i_0$ (in which case 
$j_n \geq j_0$ for large enough $n$), we have that for all $k$ with $b_n+\eta_n \leq k \leq b_{n+1}$ that
if

\[
d=\frac{1}{k-b_n}
|\{ m \in [b_n,k)  \colon E(\psi(x)(m-j),\dots,\psi(x)(m-1)) \neq \psi(x)(m)\} |
\]
\noindent
then

\begin{equation*}
\begin{split}
d  & \geq \left( \frac{b-1}{b} \right) \min \left\{ \left(1-\frac{1}{j}-\frac{j}{2^n}\right), 
\left(1-\frac{1}{n}\right)\left(\sum_{i\notin J} \frac{1}{2^{i+1}} +\sum_{\substack{i \in J\\ i>i_0}} \frac{1}{2^{i+1}}\right)
\right\}
\end{split}
\end{equation*}
\noindent
for any $j$ and all large enough $n$. 
For any $\epsilon>0$ we have for all large enough $n$ that 
$$d \geq \left( \frac{b-1}{b}\right)
\left(\sum_{i\notin J} \frac{1}{2^{i+1}} +\sum_{\substack{i \in J\\ i>i_0}} \frac{1}{2^{i+1}}\right)-\epsilon
\geq (s-\epsilon)+ \left(\frac{b-1}{b}\right)\sum_{\substack{i \in J\\ i>i_0}} \frac{1}{2^{i+1}}.$$ 

\noindent
We may assume the $b_n$ grow sufficiently fast so that $\frac{\eta_n}{b_{n-1}} \to 0$ 
and $\frac{\sum_{m<n} b_m}{b_n} \to 0$
as $n \to \infty$ (note here that $\eta_n$ is defined independently of $b_n$).
Then for large enough $n$ we have the above inequality for $d$ holds using 
\[
d'=\frac{1}{k} |\{ m \in [0,k)  \colon E(\psi(x)(m-j),\dots,\psi(x)(m-1)) \neq \psi(x)(m)\} |
\]
\noindent 
for all large enough $k$.  
Thus, $\loe(\psi(x)) \geq s+  (\frac{b-1}{b})(\sum_{\substack{i \in J\\ i>i_0}} \frac{1}{2^{i+1}})$,
and so $\psi(x) \notin A_1(s)$.

\end{proof}

\section{Hausdorff dimension of real numbers with noise $s$}

For each $s \in \bra{0,\frac{b-1}{b}}$, we know that $\mathcal{N}(b) \subset A_2(s)$ and $\mathcal{N}(b) \subset A_4(s)$ which implies $\dim_H(A_2(s)) = \dim_H(A_4(s)) = 1$. Since $A_1(s)$ and $A_3(s)$ do not contain $\mathcal{N}(b)$ we must introduce the following machinery to compute the Hausdorff dimension of these sets.

For $x \in b^\omega$ let $M(x)$ be the set of weak-$*$ limit points of the sequence of measures $\mu_{x,n} = \frac{1}{n} \sum_{k=1}^{n} \delta_{T^k x}$ where $T(.E_1 E_2 E_3 \cdots) = .E_2 E_3 \cdots$ is the shift on $b^\omega$. This is a closed convex subset of the shift-invariant probability measures on $b^\omega$ which we denote by $\MM(b^\omega)$. We say a point $x \in b^\omega$ is generic for a measure $\mu$ if $M(x) = \{\mu\}$. For a closed convex subset $M \subseteq \MM(b^\omega)$ define $\GG(M)$ to be the set of $x \in b^\omega$ such that $M(x) = M$. Note $\GG(\{\mu\})$ is the set of generic points for $\mu$. Recall the measure theoretic entropy of a shift-invariant measure $\mu$ on $b^\omega$ is
\begin{align*}
h(\mu) & = \lim_{N \to \infty} \frac{1}{N} \sum_{B \in b^N} -\mu[B] \log \mu[B].
\end{align*}
Colebrook proved  the following result in \cite{Colebrook}.
\begin{thm}
The Hausdorff dimension of $\GG(M)$ is $\sup_{\mu \in M} \frac{h(\mu)}{\log b}$.
\end{thm}

In an analogous way to real numbers we can define the noise of a measure $\mu$ as 
\begin{align*}
\beta(\mu) = 1 - \lim_{\ell \to \infty}\sum_{B \in b^\ell} \mu[B] \max_{0 \leq d \leq b-1} \frac{\mu[d \conc B]}{\mu[B]}.
\end{align*}
It is clear that the upper and lower noise of each element in $\GG(M)$ is $\overline{\beta}(M) : = \sup_{\mu \in M} \beta(\mu)$ and $\underline{\beta}(M): = \inf_{\mu \in M} \beta(\mu)$ respectively. In general
\begin{align*}
\overline{\beta}(x) &= \overline{\beta}(M(x)) \\
\underline{\beta}(x) & = \underline{\beta}(M(x)).
\end{align*}

Thus
\begin{align*}
A_1(s) &= \bigcup_{\substack{M \subseteq \MM(b^\omega) \\ M\text{ is closed and convex} \\ \beta^{-}(M) \leq s }} \GG(M) \ \ \ \ \ \ \ \ A_2(s)= \bigcup_{\substack{M \subseteq \MM(b^\omega) \\ M\text{ is closed and convex} \\ \beta^{-}(M) \geq s }} \GG(M) \\
A_3(s) &= \bigcup_{\substack{M \subseteq \MM(b^\omega) \\ M\text{ is closed and convex} \\ \beta^+(M) \leq s }} \GG(M)  \ \ \ \ \ \ \ \ A_4(s) = \bigcup_{\substack{M \subseteq \MM(b^\omega) \\ M\text{ is closed and convex} \\ \beta^+(M) \geq s }} \GG(M).
\end{align*}
Furthermore
\begin{align*}
U(s) & = \bigcup_{\substack{M \subseteq \MM(b^\omega) \\ M\text{ is closed and convex} \\ \beta^{+}(M) = s }} \GG(M)  \ \ \ \ \ \ \ \ L(s) = \bigcup_{\substack{M \subseteq \MM(b^\omega) \\ M\text{ is closed convex} \\ \beta^{-}(M) = s }} \GG(M).
\end{align*}
Now consider $\lambda$ the uniform measure on $b^\omega$ and $\delta_0$ the Dirac point mass at $(0,0,\cdots)$. Let $M$ be the convex hull of $\{\delta_0,\lambda\}$. Then $M$ is a closed convex subset of $\MM(b^\omega)$ and we have $\beta^-(M) = \beta(\delta_0) = 0$. Thus $\GG(M) \subset L(s) \subset A_1(s)$ and $\dim_H(L(s))  \geq \dim_H(\GG(M)) = \frac{h(\lambda)}{\log b} = 1$ which implies $\dim_H(L(s)) = \dim_H(A_1(s)) = 1$.

We also have the following lower bound on $\dim_H A_3(s)$
$$
\dim_H A_3(s) \geq \frac{1}{\log b} \sup_{\mu : \beta(\mu) \leq s} h(\mu).
$$
Note the same lower bound holds for $\dim_H U(s)$ since
$$
\dim_H U(s) \geq \frac{1}{\log b} \sup_{\mu : \beta(\mu) = s} h(\mu) \geq \frac{1}{\log b} \sup_{\mu : \beta(\mu) \leq s} h(\mu).
$$
This second inequality follows since if $\beta(\mu) < s$ we can find $t \in [0,1]$ such that $\beta(t\mu+(1-t)\lambda) = s$ where $\lambda$ is the uniform measure on $b^\omega$ and $h(t\mu + (1-t) \lambda) \geq h(\mu)$ since $\lambda$ is the measure of maximal entropy.
We have that $\beta(\mu) = 0$ is equivalent to $h(\mu) = 0$ and $\beta(\mu) = \frac{b-1}{b}$ is equivalent to $h(\mu) = \log b$ but we do not have a general expression for
\begin{align*}
\sup_{\mu : \beta(\mu) \leq s} h(\mu).
\end{align*}
One approach to finding this supremum is to restrict our attention to measures which are $k$-step Markov, that is measures of the form
\begin{align*}
\mu[b_1, b_2, \cdots, b_{\ell}] & = \rho(b_1, \cdots, b_k) \prod_{i=1}^{\ell - k} P_{b_i \cdots b_{i+k-1}, b_{i+1} \cdots b_{i+k}}
\end{align*}
where $\rho$ is a probability distribution on $b^k$ (which we view as a $1 \times b^k$ matrix) and $P$ is a $b^k \times b^k$ matrix of non-negative real numbers such that
\begin{align*}
&\sum_{B' \in b^k} P_{B,B'} = 1 \text{ for all } B \in b^k\\
&\rho P = \rho \\
& P(B,B')>0 \Rightarrow B_i = B'_{i-1} \text{ for } 2 \leq i \leq k.
\end{align*}
A $k$-step Markov chain $\mu$ with stationary distribution $\rho$ and transition matrix $P$ has entropy
\begin{align*}
h(\mu) =: h_k(\rho,P) = \sum_{B \in b^k} \rho(B) \sum_{B' \in b^k} - P_{B,B'} \log P_{B,B'}
\end{align*}
and noise
\begin{align*}
\beta(\mu) =: \beta_k(\rho,P) &= 1 - \sum_{B \in b^k} \mu[B] \max_{0 \leq d \leq b-1} \frac{\mu[d \conc B]}{\mu[B]} \\
&= 1 - \sum_{B \in b^k} \mu[B] P_{d b_1 \cdots b_{k-1}, b_1 b_2 \cdots b_k}.
\end{align*}
Thus $\sup_{(\rho, P) : \beta_k(\rho,P) \leq s} h_k(\rho,P) \leq \sup_{\mu : \overline{\beta}(\mu) \leq s} h(\mu)$. If we can compute this supremum over all stochastic matrices $P$ with steady state $\rho$ we have improved lower bounds on $\dim_H(U(s))$. Problems of this type are unfortunately quite difficult in general. This problem is tractable for small $k$ however, and we now consider an easy special case.
\begin{lem}
For $s \in \bra{0, \frac{b-1}{b}}$
\begin{align*}
\dim_H(A_3(s)) \geq \frac{1}{\log b} H(s) + s \frac{\log(b-1)}{\log b}
\end{align*}
where $H(s) = -s\log s - (1-s) \log (1-s)$.
\end{lem}
\begin{proof}
For $k=1$ and $P$ with identical rows (so we now write $P_j$ instead of $P_{i,j}$) we have the associated Markov measure $\mu$ is Bernoulli which implies
\begin{align*}
h(\mu) &= - \sum_{d =0}^{b-1} P_d \log P_d \\
\beta(\mu) &= 1 - \max_{0 \leq d \leq b-1} P_d.
\end{align*}
Now if $\beta(\mu) \leq s$ then $\max_{0 \leq d \leq b-1} P_d \geq 1 - s$ and we can without of loss of generality take this maximum to be at $d=0$. Therefore
\begin{align*}
\dim_H A_3(s) &\geq \sup_{\substack{P: P_0 \geq 1-s \\P_d \leq P_0 \forall 0 \leq d \leq b-1 \\ \sum_d P_d =1}} - \sum_{d=0}^{b-1} P_d \log P_d.
\end{align*}
Clearly this supremum is attained when $P_0 = 1-s$ since $1 - s \geq \frac{1}{b}$.  Thus, we must maximize $- \sum_{d =0 }^{b-1} P_d \log P_d$ subject to the constraints $P_0 = 1-s$, $0 \leq P_d \leq 1-s$ and $\sum_{d=0}^{b-1} P_d = 1$, which is equivalent to maximizing $- \sum_{d=1}^{b-1} P_d \log P_d$ subject to the constraints $0 \leq P_d \leq 1-s$ and $\sum_{d=1}^{b-1} P_d = s$. The maximizer of $-\sum_{d=1}^{b-1} P_d \log P_d$ subject to the constraint $\sum_{d=1}^{b-1} P_d = s$ occurs when $P_d = \frac{s}{b-1}$. Since $s \leq \frac{b-1}{b}$ we have $\frac{s}{b-1} \leq 1- s$, so for each $1 \leq d \leq b-1$ we have $0 \leq P_d \leq 1-s$. Thus we have
\begin{align*}
\dim_H(A_3(s)) &\geq  \frac{1}{\log b} \pr{-(1-s) \log (1-s) - \sum_{d=1}^{b-1} \frac{s}{b-1} \log \frac{s}{b-1}}\\
&= \frac{1}{\log b}\pr{- (1-s)\log(1-s) - s \log \frac{s}{b-1}}\\
& = \frac{1}{\log b} H(s) + s \frac{\log (b-1)}{\log b}.
\end{align*}
\end{proof}

In order to obtain an upper bound on $\dim_H(A_3(s))$ we use the argument of M. Bernay \cite{Bernay} which showed $\dim_H \mathcal{N}^\perp(b) = 0$. To do this define the sets
\begin{align*}
A(N_0,\ell,s,\epsilon) = \bigcap_{N \geq N_0} \bigcup_{E \in \mathcal{E}_\ell} \br{\omega \in [0,1) : \sum_{n<N} \inf \{1, |\omega_n - E(\omega_{n+1},\cdots, \omega_{n+\ell})|\} \leq N(s+\epsilon)}
\end{align*}
and note
\begin{align*}
A_3(s) = \bigcap_{\epsilon > 0} \bigcup_{\ell = 1}^\infty \bigcup_{N_0 = 1}^\infty A(N_0,\ell,s,\epsilon).
\end{align*}
M. Bernay proved the following lemma.
\begin{lem}
For $s \in \bra{0,\frac{b-1}{b}}$
$$
\dim_H A(N_0,\ell,s,\epsilon) \leq \frac{1}{\log b} H(s+\epsilon) + s+\epsilon.
$$
\end{lem}
This implies the following upper bound on $\dim_H(A_3(s))$.
\begin{lem}
For $s \in \bra{0,\frac{b-1}{b}}$
$$
\dim_H A_3(s) \leq \frac{1}{\log b} H(s) + s.
$$
\end{lem}
Thus, we have proven Theorem~\ref{HD}.

\bibliographystyle{amsplain}


\providecommand{\bysame}{\leavevmode\hbox to3em{\hrulefill}\thinspace}
\providecommand{\MR}{\relax\ifhmode\unskip\space\fi MR }
\providecommand{\MRhref}[2]{%
  \href{http://www.ams.org/mathscinet-getitem?mr=#1}{#2}
}
\providecommand{\href}[2]{#2}

\end{document}